\documentclass[10pt, reqno, a4paper]{amsart}

\usepackage[latin1]{inputenc}
\usepackage{amsfonts}
\usepackage{amsmath}
\usepackage{amssymb}
\usepackage{amsthm}
\usepackage[T1]{fontenc}
\usepackage{enumerate}
\usepackage{verbatim}
\usepackage{lmodern}

\usepackage[pdftex]{hyperref}
\usepackage[matrix, arrow]{xy}
\usepackage{listings}

\title{$V$-filtrations in positive characteristic and test modules}
\author{Axel St\"abler}

\DeclareMathOperator{\Supp}{Supp}

\DeclareMathOperator{\Hom}{Hom}

\DeclareMathOperator{\id}{id}
\DeclareMathOperator{\coker}{coker}
\DeclareMathOperator{\colim}{colim}

\DeclareMathOperator{\Ass}{Ass}
\DeclareMathOperator{\Ann}{Ann}

\DeclareMathOperator{\Spec}{Spec}

\DeclareMathOperator{\depth}{depth}
\DeclareMathOperator{\nil}{nil}

\newcommand{\eps}{\varepsilon}
\input xy
\xyoption{all}
\SelectTips{cm}{10}
\begin{document}
\swapnumbers
\theoremstyle{definition}
\newtheorem{Le}{Lemma}[section]
\newtheorem{Def}[Le]{Definition}
\newtheorem*{DefB}{Definition}
\newtheorem{Bem}[Le]{Remark}
\newtheorem{Ko}[Le]{Corollary}
\newtheorem{Theo}[Le]{Theorem}
\newtheorem*{TheoB}{Theorem}
\newtheorem{Bsp}[Le]{Example}
\newtheorem{Be}[Le]{Observation}
\newtheorem{Prop}[Le]{Proposition}
\newtheorem*{PropB}{Proposition}
\newtheorem{Sit}[Le]{Situation}
\newtheorem{Que}[Le]{Question}
\newtheorem{Con}[Le]{Conjecture}
\newtheorem{Dis}[Le]{Discussion}
\newtheorem{Prob}[Le]{Problem}
\newtheorem{Konv}[Le]{Convention}
\def\cocoa{{\hbox{\rm C\kern-.13em o\kern-.07em C\kern-.13em o\kern-.15em
A}}}

\address{Axel St\"abler\\
Johannes Gutenberg-Universit\"at Mainz\\ Fachbereich 08\\
Staudingerweg 9\\
55099 Mainz\\
Germany}
\email{staebler@uni-mainz.de}

\date{\today}
\subjclass[2010]{Primary 13A35; Secondary 14B05}
\begin{abstract}
Let $R$ be a ring essentially of finite type over an $F$-finite field.
Given an ideal $\mathfrak{a}$ and a principal Cartier module $M$ we introduce the notion of a $V$-filtration of $M$ along $\mathfrak{a}$. If $M$ is $F$-regular then this coincides with the test module filtration. We also show that the associated graded induces a functor $Gr^{[0,1]}$ from Cartier crystals to Cartier crystals supported on $V(\mathfrak{a})$. This functor commutes with finite pushforwards for principal ideals and with pullbacks along essentially \'etale morphisms. We also derive corresponding transformation rules for test modules generalizing previous results by Schwede and Tucker in the \'etale case (cf.\ \cite{schwedetuckertestidealsfinitemaps}).

If $\mathfrak{a} = (f)$ defines a smooth hypersurface and $R$ is in addition smooth then for a Cartier crystal corresponding to a locally constant sheaf on $\Spec R_{\acute{e}t}$ the functor $Gr^{[0,1]}$ corresponds, up to a shift, to $i^!$, where $i: V(\mathfrak{a}) \to \Spec R$ is the closed immersion.
\end{abstract}
\maketitle

\section*{Introduction}
In the study of (say) hypersurface singularities $f: X \to \mathbb{C}$ the theory of nearby and vanishing cycles due to Deligne \cite[Expos\'e XIII]{SGA7II} is a very important tool. Nearby and vanishing cycles are derived functors from the derived category of constructible sheaves on $X$ (or $X_{\acute{e}t}$) to the derived category of constructible sheaves on the fiber $X_0$. These come equipped with an action of the (\'etale) fundamental group.

Over the complex numbers the famous Riemann-Hilbert correspondence states that for $X$ a complex manifold there is a functor from the bounded derived category of regular holonomic $\mathcal{D}_X$-modules to the bounded derived category of constructible sheaves inducing an equivalence of category which respects the ``six operations'' $f_\ast, f^\ast, f^!, f_!, \otimes, RHom$.

By work of Malgrange (\cite{malgrangevanishingdmodule}) and Kashiwara (\cite{kashiwaravfiltration}) one has a direct construction of nearby and vanishing cycles for regular holonomic $\mathcal{D}_X$-modules without passing through the Riemann-Hilbert correspondence. The key ingredient in this construction is the so-called $V$-filtration of a (regular holonomic) $\mathcal{D}_X$-module whose associated graded pieces recover nearby and vanishing cycles. The construction of the $V$-filtration itself uses so-called Bernstein-Sato polynomials.

Due to work of Budur and Saito (cf.\ \cite{budurvfiltrationdmodules}, \cite{budursaitomultiplieridealvfiltration}) it is known that the $V$-filtration with respect to $f: X \to \mathbb{A}^1_{\mathbb{C}}$ of $\mathcal{O}_X$ considered as a $\mathcal{D}_X$-module is intimately related to the multiplier ideal filtration of $f$ from birational geometry. To be precise, if $i: X \to X \times \mathbb{A}^1_{\mathbb{C}}$ denotes the graph embedding, then the associated graded in the range $[0,1]$ of the $V$-filtration of $i_+ \mathcal{O}_X$ recovers nearby cycles, while $V \mathcal{O}_X$ is up to a renormalization the multiplier ideal filtration.

This article is concerned with a partial characteristic $p$ analog of $V$-filtrations. The characteristic $p$-version of multiplier ideals are so-called test ideals. These test ideals $\tau$ naturally carry additional structure, namely they are endowed with an algebra of maps $F_\ast^e \tau \to \tau$ for varying $e$, a so-called Cartier algebra, where $F$ denotes the Frobenius morphism. In work of Blickle (\cite{blicklep-etestideale}) this point of view is further emphasized and the construction of test ideals is extended to modules. This will be our point of departure -- we will construct a $V$-filtration for Cartier modules and show that this coincides with the test module filtration under some conditions. In particular, this means that the test ideal filtration (say in a polynomial ring) is uniquely determined by discreteness, rationality, the Brian\c{c}on-Skoda property and how the Cartier operator acts on the filtration (mapping the $pt$th graded piece to the $t$th graded piece). This also holds for modules, provided one restricts to a suitable class of modules.

Our analog of a Riemann-Hilbert correspondence is due to work of Blickle and B\"ockle (\cite{blickleboecklecartiercrystals}, \cite{blickleboecklecartierfiniteness}, \cite{blickleboecklecartierekequivalence}) and of Emerton and Kisin (\cite{emertonkisinrhunitfcrys}). Given a separated $F$-finite smooth scheme $X$ over a field of characteristic $p > 0$ the bounded derived category of $\kappa$-crystals, that is, the derived category of modules endowed with a single homomorphism $F_\ast M \to M$  up to a Serre-localization of nilpotence, is equivalent to the bounded derived category of \'etale constructible sheaves of $\mathbb{F}_p$-vector spaces. This equivalence interchanges $f^!, f_\ast$ on crystals with $f^\ast, f_!$ on constructible sheaves. Moreover, the abelian category of $\kappa$-crystals corresponds under this equivalence to perverse constructible $\mathbb{F}_p$-sheaves in the sense of \cite{gabbertstructures}. An intermediate step of this equivalence is the category of locally finitely generated unit $F$-modules by Emerton and Kisin (which are special $D$-modules). In fact, already the abelian categories of crystals and locally finitely generated unit $F$-modules are equivalent.

The formalism of nearby and vanishing cycles still exists in the context of $\mathbb{F}_p$-constructible sheaves on a scheme over $\mathbb{F}_p$ (cf.\ \cite[Expos\'e XIII]{SGA7II}), however, it is not very well-behaved. For instance, the nearby cycles of a locally constant sheaf are in general not constructible\footnote{See \cite{MO124960} for an explicit example. I learned this from Brian Conrad.}. Using the equivalence outlined above we are still able to construct a functor $Gr$ (namely the associated graded of the $V$-filtration above) on a (full) subcategory of $\kappa$-crystals that has some of the properties of nearby cycles and even preserves some of the desirable properties of nearby cycles that hold for $\mathbb{F}_\ell$-coefficients (with $\ell \neq p$ a prime) but fail if $\ell = p$.

Given an ideal $\mathfrak{a}$ in a ring $R$ and a module $M$ we say that $M$ is non-degenerate if $\mathfrak{a}$ contains an $M$-regular element. We will prove

\begin{TheoB}
Let $R$ be essentially of finite type over an $F$-finite field and $\mathfrak{a}$ a principal ideal. Let $\mathcal{A}_R$ be the category of $F$-regular coherent non-degenerate $\kappa$-crystals.
\begin{enumerate}[(i)]
\item{The associated graded of the test module filtration (in the range $[0,1]$) of $M$ induces a functor $Gr_R: \mathcal{A}_R \to \{\kappa-$crystals supported on $V(\mathfrak{a})\}$.}
 \item{If $f: \Spec S \to \Spec R$ is essentially \'etale, then $f^! Gr_R = Gr_S f^!$ and if $f$ is finite then $f_\ast Gr_S = Gr_R f_\ast$.}
\item{If in addition $R$ is regular, $\mathfrak{a}$ cuts out a smooth hypersurface and $M$ is a crystal corresponding to a locally constant sheaf on the \'etale site $(\Spec R)_{\acute{e}t}$ then $Gr M \cong i^! M[1]$, where $i:V(\mathfrak{a}) \to \Spec R$ denotes the closed immersion.}
\end{enumerate}
\end{TheoB}
\begin{proof}
(i) is obtained in Theorem \ref{GrFunctoronCrys}. Part (ii) is Theorems \ref{GrPushforwardCommute} and \ref{GrEssEtalePullbackCommute}. Finally, (iii) is Corollary \ref{MainResult?}.
\end{proof}

In order to accomplish this we will, in particular, prove a transformation rule for test modules under essentially \'etale morphisms (note that in the essentially \'etale case one has $\bullet^! = \bullet^\ast$). We call a Cartier algebra on $\Spec R$ locally principal if there is a open covering $(U_i)_i$ of $\Spec R$ such that $\mathcal{C}$ restricted to each $U_i$ is principally generated. With this definition we prove:

\begin{TheoB}
Let $R$ be a essentially of finite type over an $F$-finite field, let $\mathcal{C}$ be a locally principal Cartier algebra and assume that $f: \Spec S \to \Spec R$ is an essentially \'etale morphism. Then for $\mathfrak{a}$ an ideal and $t \in \mathbb{Q}$ we have \[f^\ast \tau(M, \mathcal{C}^{\mathfrak{a}^t}) = \tau(f^\ast M, \mathcal{C}'^{(\mathfrak{a}S)^t}),\] where  $\mathcal{C}'$ is the Cartier algebra obtained by base change from $\mathcal{C}$. If $R$ and $S$ are integral normal domains, $\Supp M = \Spec R$ and $f$ is finite \'etale then \[f_\ast \tau(f^\ast M, \mathcal{C}'^{(\mathfrak{a}S)^t}) \cap M = \tau(M, \mathcal{C}^{\mathfrak{a}^t}).\]
\end{TheoB}
\begin{proof}
This is a combination of Corollaries \ref{TestModuleLocallyPrincipal} and \ref{TestModuleEtaleTransformationNonPrincipal}. 
\end{proof}

This is a partial generalization of results obtained by Schwede and Tucker (cf. \cite{schwedetuckertestidealsfinitemaps}) for $M = R$.

Similar questions concerning the construction of $V$-filtrations have been investigated before using the category of locally finitely generated unit $F$-modules. More precisely, Musta\c{t}\u{a} constructed a sequence of Bernstein-Sato-like polynomials associated to a hypersurface and showed that the roots of these polynomials are related to the $F$-jumping numbers of the associated test ideal filtration (\cite[Theorem on p.\,130]{mustatabernsteinsatopolynomialspositivechar}). See also \cite{blicklestaeblerbernsteinsatocartier} for a generalization to Cartier modules. There is also recent work of Stadnik \cite{stadnikvfiltrationfcrystal} where he defined a $V$-filtration for a certain class of unit $F$-crystals in the setting of Emerton and Kisin and showed that the zeroth graded part of this filtration corresponds to the unipotent underived tame nearby cycles functor. In fact, our definition of $V$-filtration is inspired by that of Stadnik. However, we construct a filtration for a different class of objects and do not know how these filtrations relate.

We now describe the contents of this paper in more detail.
After recalling basic constructions for Cartier modules in Section \ref{SectionCartierModules} we give an axiomatic description of a $V$-filtration and show that it coincides, under some conditions, with the test module filtration (Theorem \ref{VfiltrationTestidealfiltration}). This is the content of sections \ref{SectionVFiltration} and \ref{SectionExistence}. However, the Cartier structure on test modules is never principle in any interesting case. In order to have an equivalence with constructible sheaves when passing to crystals we would like to have a principal Cartier structure. We will bypass this issue by considering the associated graded filtration for which we naturally obtain a principal Cartier structure (in the hypersurface case). We also show that passing to the associated graded $Gr$ is functorial. This is accomplished in section \ref{SectionAssocGraded}. In section \ref{SectionCrystals} we show that the functor $M \mapsto Gr M$ induces a corresponding functor on $\kappa$-crystals (Theorem \ref{GrFunctoronCrys}). Next we show that $Gr$ commutes with essentially \'etale pullbacks and finite pushforwards in Section \ref{SectionFunctorialProperties}. We also obtain that the test module is well-behaved along pull backs of essentially \'etale morphisms. Section \ref{SectionFUnitModules} describes the correspondence of $\kappa$-crystals with constructible sheaves in more detail. In particuar, we analyze which $\kappa$-crystals correspond to locally constant sheaves and show that our notion of $V$-filtration applies for these crystals. Finally, in section \ref{SectionMainResult} we prove that for a smooth hypersurface $f$ in regular ring $R$ essentially of finite type over an $F$-finite field the functor $Gr$ restricted to $F$-regular crystals that correspond to locally constant sheaves is, up to shift, $i^!$, where $i: \Spec R/(f) \to \Spec(R)$ (Corollary \ref{MainResult?}). This will be accomplished by proving some structural results for test module filtrations.

\subsection*{Conventions}
Commutative rings considered in this article are assumed to be noetherian and to contain a field of characteristic $p > 0$. We denote by $F: R \to R$ the Frobenius homomorphism. For a ring $R$ and an $R$-module $M$ we denote by $F^e_\ast M$ the $R$-module whose underlying abelian group is $M$ but with multiplication given by $r \cdot m = r^{p^e} m$. Also recall that a ring is called \emph{$F$-finite} if $F_\ast R$ is a finite $R$-module (in other words, the Frobenius morphism on $\Spec R$ is a finite map).
For a ring $R$ the symbol $R^\circ$ denotes elements of $R$ not contained in any minimal prime. For an ideal $I$ in a ring $R$ we denote by $I^{[p^e]}$ the ideal given by $(f^{p^e} \, \vert \, f \in I)$.

\subsection*{Acknowledgements}
I thank Manuel Blickle, Brian Conrad, Mircea Musta\c{t}\u{a}, Karen Smith and Florian Strunk for useful discussions and their interest. Moreover, I thank Manuel Blickle for a careful reading of an earlier draft and for sending me his preprints \cite{blickleboecklecartiercrystals}, \cite{blickleboecklecartierduality}. Finally, I thank the referee for corrections and useful suggestions. The author was partially supported by SFB/Transregio 45 Bonn-Essen-Mainz financed by Deutsche Forschungsgemeinschaft.

\section{Cartier modules}
\label{SectionCartierModules}
This section is concerned with basic definitions of Cartier modules and Cartier crystals. Our main reference is \cite{blicklep-etestideale}.

\begin{Def}
Let $R$ be a ring. An \emph{$R$-Cartier algebra} is an $\mathbb{N}$-graded $R$-algebra $\mathcal{C} = \bigoplus_{n \in \mathbb{N}} \mathcal{C}_n$ such that the structural morphism $R \to \mathcal{C}$ surjects onto $\mathcal{C}_0$ and such that we have the relation $r \varphi_n = \varphi_n r^{p^n}$ for all $\varphi_n \in \mathcal{C}_n$ and $r \in R$. A $\mathcal{C}$-module is also called a Cartier module. A Cartier module is called \emph{coherent} if it is finitely generated as an $R$-module. We denote $\bigoplus_{n \geq 1} \mathcal{C}_n$ by $\mathcal{C}_+$.
\end{Def}

We will only consider left $\mathcal{C}$-modules.
The key example the reader should have in mind is the following
\begin{Bsp}
Let $R$ be an $F$-finite Ring, $M$ an $R$-module and $\kappa: F_\ast M \to M$ an $R$-linear map. Fix an ideal $\mathfrak{a}$ in $R$ and $t \in \mathbb{Q}$. Then $\bigoplus_n \kappa^n \mathfrak{a}^{\lceil t (p^n -1)\rceil}$ defines a Cartier algebra. Slightly more general, if $\mathcal{C}$ is a Cartier algebra and $M$ is a $\mathcal{C}$-module then we denote the Cartier algebra $\bigoplus_n \mathcal{C}_n \mathfrak{a}^{\lceil t (p^n-1)\rceil}$ by $\mathcal{C}^{\mathfrak{a}^t}$. If $\mathcal{C}$ is principally generated with generator $\kappa$ then we also denote this algebra by $\kappa^{\mathfrak{a}^t}$.
\end{Bsp}

Cartier modules are well-behaved with respect to localization at a multiplicative subset $S \subseteq R$ -- cf.\ \cite[Lemma 2.11]{blicklep-etestideale}.

A Cartier module $M$ is called \emph{$F$-pure} if $\mathcal{C}_+ M = M$. In the case of a principal algebra generated by $\kappa: F_\ast M \to M$ this just means that $\kappa$ is surjective. If $M$ is any coherent Cartier module then $\underline{M} := \mathcal{C}_+^e M$ is $F$-pure for $e \gg 0$ (\cite[Corollary 2.14]{blicklep-etestideale}). Moreover, a coherent $\mathcal{C}$-module $M$ is called \emph{$F$-regular} if $M$ is $F$-pure and if there is no proper non-zero $\mathcal{C}$-submodule $N$ of $M$ that agrees with $M$ in each generic point of $\Supp M$ (where $\Supp M$ is the support of the underlying $R$-module). We note that if $M = R$ and $\mathcal{C}$ is the Cartier algebra generated by $\Hom(F_\ast R, R)$ then $F$-regularity for $R$ in this terminology is often called \emph{strong} $F$-regularity in the literature. We define the support of a Cartier module as the support of the underlying $R$-module.

We recall the definition of test modules as in \cite[Definition 3.1]{blicklep-etestideale}.

\begin{Def}
\label{DefTestModule}
The test module $\tau(M, \mathcal{C})$ of a pair $(M, \mathcal{C})$ is the smallest submodule of $M$ such that for every generic point $\eta$ of $\Supp \underline{M}_\mathcal{C}$ one has $\tau(M, \mathcal{C})_\eta = (\underline{M}_{\mathcal{C}})_\eta$.
\end{Def}

An immediate consequence of this definition is that $\tau(M, \mathcal{C}) = M$ if and only if $M$ is $F$-regular. We will also occasionally omit the algebra $\mathcal{C}$ from the notation and simply write $\tau(M)$ for $\tau(M, \mathcal{C})$.
Typically we have a fixed Cartier algebra $\mathcal{C}$ and an ideal $\mathfrak{a}$. We then may consider for $t \in \mathbb{Q}_{\geq 0}$ the test module associated to each of the Cartier algebras $\mathcal{C}^{\mathfrak{a}^t} = \bigoplus \mathcal{C} \mathfrak{a}^{\lceil t(p^e -1) \rceil}$ we denote this test module by $\tau(M, \mathcal{C}, \mathfrak{a}^t)$ or $\tau(M, \mathcal{C}^{\mathfrak{a}^t})$. By \cite[Proposition 4.16]{blicklep-etestideale} this filtration is right-continuous\footnote{Meaning we get the same filtration for $t$ and $t + \eps$ for $0 < \eps \ll 1$.} (provided $R$ is $F$-finite and assuming the existence of the filtration). For a given Cartier module $M$ and ideal $\mathfrak{a}$ any $t$ such the test module filtrations for $t$ and $t - \eps$ do not coincide for all $\eps > 0$ is called an \emph{$F$-jumping number}.

Very important for applications is the following
\begin{Theo}[{\cite[Theorem 3.11]{blicklep-etestideale}}]
Let $\mathcal{C}$ be an $R$-Cartier algebra and $M$ coherent $\mathcal{C}$-module. Then the test module $\tau(M, \mathcal{C})$ exists if and only if on some open subset $\Spec R_c$ for $c \in R$ whose intersection with $\Supp \underline{M}_c$ is dense in $\Supp \underline{M}_{c}$ the pair $(\underline{M_c}, \mathcal{C}_C)$ is $F$-regular.

In this setting, the test module $\tau(M, \mathcal{C})$ is given by $\sum_{e \geq 1} \mathcal{C}_e c^d \underline{M}$, where $d \in \mathbb{N}$ may be chosen arbitrarily.
\end{Theo}

We call any $c$ as in the theorem a \emph{test element} (with respect to the pair $(M, \mathcal{C})$).

\begin{Bem}
\begin{enumerate}[(a)]
\item{For the existence of test modules see \cite{blicklep-etestideale}. In particular, test modules are known to exist if $M$ is a coherent $\mathcal{C}$-module and $R$ is essentially of finite type over an $F$-finite field. Moreover, it is known in this case that the test module filtration is discrete (cf.\ \cite[Corollary 4.19]{blicklep-etestideale}).}
\item{There is also a slightly modified version of the definition of test ideal which only requires that $\tau(M, \mathcal{C})_\eta = (\underline{M}_{\mathcal{C}})_\eta$ for each generic point $\eta$ of $\Spec R$. One distinct advantage of this definition is that the operation of taking test ideals is then functorial. However, in this case the test ideal is always zero if the module is supported on a proper closed subset.
The reason for choosing the above definition is that one automatically has a Kashiwara type equivalence (cf.\ \cite[Proposition 3.2 (f)]{blicklep-etestideale}).}
\end{enumerate}
\end{Bem}

The main references for the following discussion are \cite{blickleboecklecartierfiniteness} and \cite{blickleboecklecartiercrystals}. See also \cite[Chapter 2]{boecklepinktaucrystals} for a concise discussion of localization at subcategories.
From now on we restrict ourselves to $\kappa$-modules, that is, coherent $R$-modules $M$ endowed with a single Cartier linear operator $\kappa: F_\ast M \to M$.

We call a $\kappa$-module $M$ \emph{nilpotent} if $\kappa^e M = 0$ for some $e \geq 0$. Given a morphism $\varphi: M \to N$ of $\kappa$-modules we say that $\varphi$ is a \emph{nil-isomorphism} if both $\coker \varphi$ and $\ker \varphi$ are nilpotent. The category of \emph{$\kappa$-crystals} (or \emph{Cartier crystals}) is the category obtained from coherent $\kappa$-modules by localization at nil-isomorphisms. Since nilpotent objects form a Serre-subcategory the resulting category of crystals is again abelian.

Given a $\kappa$-module $M$ there exists a unique $\kappa$-submodule $\underline{M}$ with $M/\underline{M}$ nilpotent and for which the structural map $\kappa: F_\ast \underline{M} \to \underline{M}$ is surjective (cf.\ \cite[Corollary 2.15]{blickleboecklecartierfiniteness}). In fact, $\underline{M} = \kappa^e(M)$ for $e \gg 0$. Similarly, there exists a unique nilpotent $\kappa$-submodule $M_{\nil}$ of $M$ such that $\overline{M} = M/M_{\nil}$ contains no non-zero nilpotent Cartier submodule [ibid., Proposition 2.12]. One has $\underline{(\overline{M})} = \overline{(\underline{M})}$ which we also denote by $M_{\min}$ and call the \emph{minimal} Cartier module associated to $M$. The category of minimal Cartier modules is equivalent to the category of Cartier crystals (cf.\ \cite[Theorem 3.12]{blickleboecklecartierfiniteness}). We define the \emph{support} of a $\kappa$-crystal $M$ as follows. Let $N$ be any $\kappa$-module whose associated crystal is $M$. Then we set $\Supp(M) = \Supp(\underline{N})$. Equivalently, it is the set of $P \in \Spec R$ such that $N_P$ is not nilpotent. We refer to \cite[Section 3.3]{blickleboecklecartierfiniteness} for further discussion. In particular, the authors show that this definition does not depend on the chosen representative $N$.

If $f: \Spec S \to \Spec R$ denotes a finite morphism of affine schemes then one has a functor $f_\ast$ from the category of $\kappa$-modules over $S$ to the category of $\kappa$-modules over $R$ and, for $f$ essentially \'etale or finite, a functor $f^!$ from $\kappa$-modules over $R$ to $\kappa$-modules over $S$. Concretely, if $f$ is essentially \'etale then $f^! = f^\ast$ on the underlying module categories and if $f$ is finite flat then $f^!$ is given by $\Hom(f_\ast S, -)$ -- see Lemma \ref{CartierStructureEtalePullback} and the discussion before Lemma \ref{SupportShriekFlat} for the induced Cartier structures. In general, $f^!$ is a pseudofunctor on a suitable derived category of crystals (cf.\ \cite[Section 3.3]{blickleboecklecartiercrystals} for a detailed exposition) and in this case $f^!$ will be given by $RHom(f_\ast S, -)$ for finite $f$. The functor $f_\ast$ is just $f_\ast$ on the underlying module categories and the Cartier structure is induced by $f_\ast(\kappa)$ and the fact that $f_\ast F_\ast$ is naturally isomorphic to $F_\ast f_\ast$.

The equivalence with locally finitely generated unit $F$-modules (and hence with constructible sheaves) is compatible with these functors. The functor $f^!$ corresponds to $f^\ast$ on constructible sheaves while $f_\ast$ corresponds to $f_+$. We will describe this in more detail in section \ref{SectionFUnitModules}.

\begin{Bem}
In \cite[Lemma 2.10]{blicklep-etestideale} it is pointed out that one has a similar definitions of nil-isomorphisms for modules over a Cartier algebra and that these still form a Serre-subcategory. Hence, one can still consider crystals in this case. However, this theory has not been worked out yet. Moreover, it seems unlikely that such a category is equivalent to constructible sheaves on the \'etale site. Most important for us is that in our considerations the Cartier algebra will actually vary for $t \in \mathbb{Q}_{\geq 0}$ and so we would have to deal with crystals in different categories. This can all be avoided by restricting to the hypersurface case as we will see.
\end{Bem}

\section{The $V$-filtration}
\label{SectionVFiltration}
In this section we give our definition of a $V$-filtration. This definition mimics recent work of Stadnik (\cite{stadnikvfiltrationfcrystal}), where he constructs a $V$-filtration for a certain class of locally finitely generated unit $F$-modules in the sense of Emerton and Kisin (cf.\ \cite{emertonkisinrhunitfcrys}).

\begin{Def}
\label{DefVfiltration}
Let $(M, \kappa)$ be a principal Cartier module and $\mathfrak{a}$ an ideal in $R$. Then a $V$-filtration on $M$ (along $\mathfrak{a}$) is a decreasing, \emph{right}-continuous\footnote{meaning that given $t$ for all sufficiently small $\eps > 0$ one has $V^{t} = V^{t+ \eps}$} $\mathbb{Q}$-indexed exhaustive filtration $V^\cdot$ of $M$ such that
\begin{enumerate}[(i)]
\item{$V^{0}$ is coherent and the filtration is continuous in zero.}
\item{There is an element $0 \neq x \in \mathfrak{a}$ such that multiplication with $x^j$ induces an \emph{injective} map $V^{t} \to V^{t+j}$.}
\item{If $\mathfrak{a}$ is generated by $n$ elements then $V^t = \mathfrak{a}^{n} V^{t-n}$ for $t > n$.}
\item{$\kappa(V^{tp}) = V^{t}$.}
\end{enumerate}
\end{Def}

\begin{Bem}
In characteristic zero the $V$-filtration is left-continuous and not right-continuous. We decided to reverse this here since our main example of $V$-filtrations (and the language we will use throughout this paper) are test module filtrations which are right-continuous. Obviously this is only a cosmetic choice.
\end{Bem}

\begin{Le}
\label{LeZQvalued}
Let $V^{\cdot}$ and $W^\cdot$ be two $V$-filtrations for a principal Cartier module $(M, \kappa)$. Assume that there is $k \in \mathbb{Z}$ such that for all $n \in \mathbb{Z}$ we have $V^{kn} = W^{kn}$. Then $V^\alpha = W^\alpha$ for all $\alpha \in \mathbb{Q}$.
\end{Le}
\begin{proof}
Let $\alpha \in \mathbb{Q}$ and $m \in V^\alpha$. Let $\beta \in \mathbb{Q}$ such that $m \notin W^\beta$. We have to show that $\beta > \alpha$. Assume that $\alpha \geq \beta$. If $\alpha = \beta$ then we may replace $\alpha$ by $\alpha + \varepsilon$ for $\varepsilon$ small enough by right continuity. Hence, we may assume that $\alpha > \beta$.

Then there is $n \gg 0$ and $x \in V^{\alpha p^n}$ such that $\kappa^n(x) = m$ by (iv) of Definition \ref{DefVfiltration}. By choosing $n$ sufficiently large we may assume that there is an $l \in k\mathbb{Z}$ with $\alpha p^n > l > \beta p^n$. Since $V^\cdot$ and $W^\cdot$ are decreasing we obtain $x \in V^{\alpha p^n} \subseteq V^{l} = W^{l} \subseteq W^{\beta p^n}$. In particular $\kappa^n(x) = m \in W^\beta$. This is a contradiction.
\end{proof}

\begin{Prop}
\label{VfiltrationUnique}
The $V$-filtration of a principal Cartier module is unique.
\end{Prop}
\begin{proof}
Let $M$ be a principal Cartier module with two filtrations $V^{\cdot}$ and $W^\cdot$ satisfying the axioms of Definition \ref{DefVfiltration}. Our first goal is to show that $V^{0} = W^{0}$. Let $m \in V^{0}$. Since $V^{0}$ is coherent and $W^{\cdot}$ is decreasing, there is a maximal $l \in \mathbb{Z}$ such that $V^{0} \subseteq W^l$. By property (iv) there is a sequence $m_n \in V^{0}$ such that $\kappa^{n}(m_n) = m$. We conclude that $m \in W^{l/p^n}$ for all $n$ and since $W^{\cdot}$ is assumed to be continuous in zero we obtain $m \in W^{0}$.

Let now $j \in \mathbb{Z}$ be positive and $m \in V^{-j}$. Let $s \in \mathbb{Q}$ be minimal with $m \notin W^{s}$. Using axiom (ii) we find $x \in \mathfrak{a}$ such that $x^j m \in V^{0} = W^{0}$. By injectivity we have $x^j m \notin W^{s + j}$. Hence, we obtain $0 < s+ j$ and $m \in W^{-j}$.
Together with axiom (iii) and Lemma \ref{LeZQvalued} this yields that $W^{s} = V^{s}$ for all $s \in \mathbb{Q}$.
\end{proof}

\begin{Bem}
One can also define a $V$-filtration that is $\mathbb{R}$-indexed. The statements and proofs above transfer to this seemingly more general situation. Moreover, if $t \in \mathbb{R}$ is such that $V^{t} \neq V^{t - \eps}$ for all $\eps > 0$ then $t$ is rational if $t > 0$. This follows similarly to the rationality of $F$-jumping numbers of test ideals as in \cite{blicklemustatasmithdiscretenessrationality}. More precisely, an analog of \cite[Proposition 3.4]{blicklemustatasmithdiscretenessrationality} is easily derived from axioms (iii) and (iv) then one can proceed as in \cite[Theorem 3.1]{blicklemustatasmithdiscretenessrationality}. I do not know if the same holds true for $t < 0$. As the $V$-filtrations we consider will actually have no non-trivial negative part we feel justified to only consider $\mathbb{Q}$-indexed filtrations. Also cf.\ Remark \ref{RemarkOnNegativeComponents} below.
\end{Bem}

\section{Existence}
\label{SectionExistence}
In this section we show that if $R$ is essentially of finite type over an $F$-finite field and $(M, \kappa)$ is an $F$-regular Cartier module such that $\mathfrak{a}$ contains an $M$-regular element\footnote{this last condition is in particular satisfied if $\depth(\mathfrak{a}, M) \geq 1$} then the filtration of test submodules $(\tau(M, \kappa, \mathfrak{a}^t))_{t \geq 0}$ (as defined in \cite{blicklep-etestideale} or \ref{DefTestModule}) is a $V$-filtration of $M$. In particular, the $V$-filtration exists in this case. We will also see that if $M$ is not $F$-regular then the $V$-filtration and the test module filtration do not coincide.

\begin{Le}
\label{LeAltCartierAlgebra}
Let $R$ be a ring, $\mathfrak{a}$ an ideal and $\mathcal{C}$ a Cartier algebra. Fix a coherent $\mathcal{C}$-module $M$.
Let $\mathcal{D} = \sum_{e \geq 0} \mathcal{C}_e^{\mathfrak{a}^t} = \sum_{e \geq 0}\mathcal{C}_e \cdot \mathfrak{a}^{\lceil t(p^e -1)\rceil}$ and $\mathcal{E} = R + \sum_{e \geq 1}\mathcal{C}_e \cdot \mathfrak{a}^{\lceil tp^e \rceil}$. Then $\tau(M, \mathcal{D}) = \tau(M, \mathcal{E})$.
\end{Le}
\begin{proof}
Arguing as in \cite[proof of Proposition 4.16]{blicklep-etestideale} we may assume that there is $c \in R^{\circ} \cap \mathfrak{a}$ such that $\tau(M, \mathcal{D}) = \mathcal{D}_+ cM$ and that likewise $\tau(M, \mathcal{E}) = \mathcal{E}_+ cM$. The inclusion $\tau(M, \mathcal{E}) \subseteq \tau(M, \mathcal{D})$ is then trivial.
For the other inclusion we may replace $c$ with $ca$ where $a \in \mathfrak{a}^{\lceil t \rceil} \cap R^\circ$. Then we have, using \cite[Theorem 3.11]{blicklep-etestideale} \[\tau(M, \mathcal{D}) = \sum_{e \geq 1} \mathcal{C}_e \mathfrak{a}^{\lceil t(p^e -1)\rceil} c^2 M \subseteq \sum_{e \geq 1} \mathcal{C}_e \mathfrak{a}^{\lceil t(p^e)\rceil} c M \subseteq \tau(M, \mathcal{E}).\]
\end{proof}

\begin{Prop}
\label{FiltrationsProp1}
Let $R$ be a ring. Fix an $R$-Cartier algebra $\mathcal{C}$ that is generated in degree $1$, a coherent $\mathcal{C}$-module $M$ and an ideal $\mathfrak{a}$. Assume that the test modules $\tau(M, \mathcal{C}, \mathfrak{a}^t)$ exist for $t \in \mathbb{Q}_{\geq 0}$. Then $\tau(M, \mathcal{C}, \mathfrak{a}^t) = \mathcal{C}_1(\tau(M, \mathcal{C}, \mathfrak{a}^{tp}))$.
\end{Prop}
\begin{proof}
First of all, we may assume that the pair $(M, \mathcal{C}^{\mathfrak{a}^{t}})$ is $F$-pure and that $R$ is reduced (use \cite[Proposition 3.2]{blicklep-etestideale}). In particular, the claim holds for $t = 0$.

We assume $t > 0$ and start with the inclusion from right to left. By Lemma \ref{LeAltCartierAlgebra} we may replace the algebra $\mathcal{C}^{\mathfrak{a}^t}$ by $\mathcal{D} =  R + \sum_{e \geq 1} \mathcal{C}_e \mathfrak{a}^{\lceil tp^e\rceil}$ and similarly for $\mathcal{C}^{\mathfrak{a}^{tp}}$. Moreover, as in \cite[Proposition 4.16]{blicklep-etestideale} we may assume that there is $c \in R^{\circ} \cap \mathfrak{a}$ such that $\tau(M, \mathcal{D}^{\mathfrak{a}^{t + \eps}}) = \mathcal{D}_+^{\mathfrak{a}^{t + \eps}} cM$ for all $\eps \geq 0$.

We therefore have \begin{align*}
\mathcal{C}_1(\tau(M, \mathcal{C}, \mathfrak{a}^{tp})) &=  \mathcal{C}_1 \sum_{e\geq 1} \mathcal{C}_e \mathfrak{a}^{\lceil tp \cdot p^e\rceil} c M = \sum_{e \geq 2} \mathcal{C}_e \mathfrak{a}^{\lceil tp^e\rceil}cM \text{ and}\\
\tau(M, \mathcal{C}, \mathfrak{a}^{t}) &= \sum_{e \geq 1} \mathcal{C}_e \mathfrak{a}^{\lceil tp^e\rceil}cM.
\end{align*}
In particular, the inclusion from right to left holds. For the inclusion from left to right we show that
\[\sum_{e \geq 1} \mathcal{C}_e \mathfrak{a}^{\lceil t(p^e -1)\rceil} c M \subseteq \sum_{e \geq 1} \mathcal{C}_{e+1} \mathfrak{a}^{\lceil tp(p^e -1) \rceil} c M.\]
For $e \geq 2$ it is enough to show $\mathcal{C}_e \mathfrak{a}^{\lceil t(p^e-1)\rceil}cM \subseteq \mathcal{C}_e \mathfrak{a}^{\lceil tp(p^{e-1} -1)\rceil} c M$. This holds since the right-hand side is equal to $\mathcal{C}_e \mathfrak{a}^{\lceil t(p^e-1) - t(p-1)\rceil} cM$. For $e =1$ we may replace $c$ on the right-hand side by $c^{p}$ (cf.\ \cite[Theorem 3.11]{blicklep-etestideale}) then we have \[\mathcal{C}_{1}^2 \mathfrak{a}^{\lceil t p (p -1)\rceil} c^{p} M \supseteq \mathcal{C}_{1}^2 \mathfrak{a}^{\lceil t (p -1)\rceil p} c^{p} M \supseteq \mathcal{C}_{1}^2 \mathfrak{a}^{\lceil t (p -1)\rceil [p]} c^{p} M  \supseteq \mathcal{C}_1 \mathfrak{a}^{\lceil t(p^e-1)\rceil} c \mathcal{C}_{1} M.\] Note that $\mathcal{C}_1 M= M$ since $M$ is $F$-pure. 
\end{proof}

\begin{Theo}
\label{VfiltrationTestidealfiltration}
Let $R$ be a ring essentially of finite type over an $F$-finite field and $\mathfrak{a}$ an ideal. Fix a principal Cartier algebra $\mathcal{C}$ generated by $\kappa$ and an $F$-regular coherent Cartier module $M$ and assume that $\mathfrak{a}$ contains an $M$-regular element. Then the $V$-filtration on $M$ exists and $V^t = \tau(M, \mathcal{C}, \mathfrak{a}^{t})$.
\end{Theo}
\begin{proof}
By \cite[Theorem 4.13]{blicklep-etestideale} the test module $\tau(M, \mathcal{C}, \mathfrak{a}^t)$ exists in this setting.
We verify that the test module filtration has the properties listed in \ref{DefVfiltration}. By Proposition \ref{VfiltrationUnique} it then coincides with the $V$-filtration.
First of all, since $M$ is $F$-regular $\tau(M, \mathcal{C}) = M$. Hence, the filtration is exhaustive (and we set $V^t = M$ for $t < 0$) and continuous in zero. By \cite[Proposition 4.16]{blicklep-etestideale} it is decreasing and right-continuous. Property (ii) is satisfied since we assume that $\mathfrak{a}$ contains an $M$-regular element and by the Brian\c{c}on-Skoda Theorem for test modules (cf.\ \cite[Theorem 4.21]{blicklep-etestideale}).

Property (iii) is a weaker version of the Brian\c{c}on-Skoda Theorem and the last property follows from Proposition \ref{FiltrationsProp1}.
\end{proof}

Since we will frequently make use of the assumption that $\mathfrak{a}$ contains an $M$-regular element we make this into a definition:
\begin{Def}
Given a ring $R$, an ideal $\mathfrak{a}$ and an $R$-module $M$ we say that $M$ is \emph{non-degenerate (with respect to $\mathfrak{a}$)} if $\mathfrak{a}$ contains an $M$-regular element.
\end{Def}

\begin{Bem}
Assume that $M = R$ is reduced and assume that $\kappa: F_\ast R \to R$ is non-zero when localized at minimal primes (this automatically holds if $(R, \kappa)$ is $F$-pure). Since $R$ is reduced the union of the minimal primes of $R$ is precisely the set of zero divisors in $R$. Hence, $\mathfrak{a}$ contains an $R$-regular element if and only if for every $t \geq 0$ there is $\varphi \in \kappa^{\mathfrak{a}^t}$  such that $\varphi: F_\ast^e R_{\eta} \to R_\eta$ is nonzero for all minimal primes $\eta$ of $R$. This condition on $\kappa^{\mathfrak{a}^t}$ is called non-degenerate in \cite[Definition 3.8]{schwedenonqgorenstein}.
\end{Bem}

\begin{Bem}
 \label{RemarkOnNegativeComponents}
It is in general not true that the test module filtration and the $V$-filtration coincide. For instance, take $M = R$ and endow it with a Cartier structure that is $F$-pure but not $F$-regular (a concrete example is $R = \mathbb{F}_p[x]$ and $\kappa = C x^{p-1}$, where $C$ is a generator of $\Hom(F_\ast R, R)$). The $V$-filtration of $R$ with respect to the unit ideal is then the trivial filtration on $R$ (i.\,e.\ $V^t R = R$ for all $t$). Indeed, $\kappa$ is surjective by $F$-purity so that condition (iv) is satisfied, conditions (ii) and (iii) are vacuous since $1 \in \mathfrak{a}$ and condition (i) is also satisfied. However, this is not the test module filtration since $\tau(R) = R$ if and only if $R$ is $F$-regular with respect to $\kappa$.

Similarly, whenever $M$ is not $F$-regular $\tau(M, \kappa) \neq M$ so that the test module filtration is not exhaustive and hence cannot coincide with the $V$-filtration (which we do not know to exist beyond the trivial example above and the result of Theorem \ref{VfiltrationTestidealfiltration}).

In order to extend the characterization of test module filtrations as $V$-filtrations to not necessarily $F$-regular modules one probably has to allow non-trivial negative components for $V$-filtrations. One should be able to make sense of this by using fractional ideals in a suitable way. We expect that the continuity in zero (i.\,e.\ (i) in Definition \ref{DefVfiltration}) has to be weakened as well.

Another approach (at least if $M$ is $F$-pure), as the referee suggests, may be to use a concept similar to the one of Takagi's \emph{adjoint ideals} -- see \cite{takagicharpadjointideals}.

Under the assumption that test modules exist axioms (iii) and (iv) hold quite generally (cf.\ \cite[Theorem 4.21]{blicklep-etestideale} and Proposition \ref{FiltrationsProp1} above) for the test module filtration. Axioms (i) and (ii) are essentially conditions on the module $M$. However, discreteness (and rationality) of the test module filtration are, to my knowledge, not known beyond the case of $R$ being essentially of finite type over an $F$-finite field and existence of test modules in general is an open problem.
\end{Bem}

In the characteristic zero setting the construction of the $V$-filtration is reduced to the situation where $\mathfrak{a}$ defines a smooth subvariety (cf.\ \cite{budurvfiltrationdmodules}). This is accomplished using a graph construction. More precisely, let $i: \Spec R \to \Spec R \times_k \mathbb{A}^n_k$ be the graph morphism with respect to generators\footnote{One can show that this construction is independent of the chosen set of generators.} $(f_1, \ldots, f_n)$ of $\mathfrak{a}$ -- explicitly $i$ is given by the map $R[t_1, \ldots, t_r] \to R, t_i \mapsto f_i$. Then the $V$-filtration of a regular holonomic and quasi-unipotent $\mathcal{D}_R$-module $M$ along $\mathfrak{a}$ is defined to be the $V$-filtration of  $i_+ M$ along $\Spec R \times \{0\}$ intersected with $M \otimes 1$ (here $i_+$ denotes the $\mathcal{D}$-module pushforward). Note however that the connection with nearby and vanishing cycles in the hypersurface case is obtained by considering the associated graded of $V(i_+ M)$.

In positive characteristic there is a direct construction even if $\mathfrak{a}$ defines a singular subvariety. One may also, however, reduce to the smooth case as in characteristic zero:

Let $R$ be a ring, $\mathcal{C}$ an $R$-Cartier algebra and $M$ a coherent $F$-pure $\mathcal{C}$-module. Recall (\cite[Proposition 2.21]{blicklep-etestideale}) that if the support of $M$ is contained in $\Spec R/J$ for some ideal $J$ in $R$ then $M$ is naturally a $\mathcal{C}/\mathcal{C}J$-module. In the following let $i: \Spec S/J \to \Spec S$ be a closed immersion. Then the pushforward $i_\ast M$ is an $S$-Cartier module since $\mathcal{C}$ is via $S \to S/J$ naturally an $S$-Cartier algebra.

\begin{Prop}
\label{VFiltrationGraph}
Let $R$ be a ring essentially of finite type over an $F$-finite field $k$, $\mathfrak{a} = (f_1, \ldots, f_r)$ an ideal and let $M$ be a coherent Cartier module. Denote the graph morphism by $i: \Spec R \to \Spec S = \Spec R \times_k \mathbb{A}_{k}^r$ and $\mathfrak{b} = (t_1, \ldots, t_r)$. Then the support of $\tau = \tau(i_\ast M, \mathcal{C}, \mathfrak{b}^t)$ is contained in $V(I)$ with $I = (t_1- f_1, \ldots, t_r - f_r)$ and $\tau$ considered as an $\mathcal{C}^{\mathfrak{b}^t}/ \mathcal{C}^{\mathfrak{b}^t} I$-module coincides with $\tau(M, \mathcal{C}, \mathfrak{a}^t)$.
\end{Prop}
\begin{proof}
By construction $I \subseteq \Ann(i_\ast M)$. It follows that $\Supp \tau$ is contained in $V(I)$ and by the above $\tau$ is a $\mathcal{C}^{\mathfrak{b}^t}/ \mathcal{C}^{\mathfrak{b}^t} I$-module. Also note that $\mathcal{C}^{\mathfrak{b}^t}/ \mathcal{C}^{\mathfrak{b}^t} I = \mathcal{C}^{\mathfrak{a}^t}$. Since $\Supp(\underline{i_\ast M})\subseteq \Supp(i_\ast M)$ and $\tau = \tau(i_\ast M, \mathcal{C}^{\mathfrak{b}^t}) = \tau(\underline{i_\ast M}, \mathcal{C}^{\mathfrak{b}^t})$ we obtain from \cite[Proposition 3.2 (f)]{blicklep-etestideale} that $\tau(i_\ast M, \mathcal{C}^{\mathfrak{b}^t}) = \tau(M, \mathcal{C}^{\mathfrak{a}^t})$.
\end{proof}

In particular, we \emph{cannot} distinguish $M$ and $i_\ast M$ even if we consider them as $R$-modules rather than Cartier modules. This is in stark contrast to the $\mathcal{D}$-module situation in characteristic zero. In fact, in characteristic zero the associated graded (in the range $[0,1]$) of the $V$-filtration of $i_+ \mathcal{O}_X$ recovers nearby cycles, while $(\mathcal{O}_X \otimes 1) \cap V^t(i_+ \mathcal{O}_X)$ recovers the multiplier ideal filtration (cf.\ \cite[Theorem 0.1]{budursaitomultiplieridealvfiltration}). In this sense we cannot expect to obtain more than the test module filtration when constructing a $V$-filtration in characteristic $p >0$.

\section{The associated graded Cartier module of the test module filtration}
\label{SectionAssocGraded}
We define the associated graded module of the $V$-filtration, and in the situation of Theorem \ref{VfiltrationTestidealfiltration} we define a Cartier structure on it and show that the assignment $M \to Gr^\cdot M$ is functorial. Some of the properties we discuss here also follow from the axioms describing a $V$-filtration. However, we mostly use the language of test modules since this is our main application and probably more familiar to the reader.

\begin{Le}
\label{SupportFPure}
Let $R$ be a ring, $\mathfrak{a} \subseteq R$ an ideal, $t \in \mathbb{Q}_{\geq 0}$, $\mathcal{C}$ a Cartier algebra and $M$ a non-degenerate coherent $F$-pure $\mathcal{C}$-module. Then $\Supp \underline{M}_{\mathcal{C}^{\mathfrak{a}^t}} = \Supp M$.
\end{Le}
\begin{proof}
By assumption there is $a \in \mathfrak{a}$ such that multiplication by $a: M \to M$ is injective. Fix $x \in \Supp M$. Since localization is exact $a: M_x \to M_x$ is then also injective. Since $F$-purity localizes (cf.\ \cite[Lemma 2.18]{blicklep-etestideale}) we may assume that $M = M_x$ and $R = R_x$. Denote the algebra $\mathcal{C}^{\mathfrak{a}^t}$ by $\mathcal{D}$.

We have to show that $\mathcal{D}_+^e M \neq 0$ for $e \gg 0$. We observe that \[\mathcal{C}_s \mathfrak{a}^{\lceil t p^s \rceil} \supseteq \mathcal{C}_s \mathfrak{a}^{\lceil t \rceil n p^s} \supseteq \mathfrak{a}^{\lceil t \rceil n} \mathcal{C}_s\] for all $s, n \geq 1$. For $e \gg 0$ as above there is $\varphi$ in $\mathcal{C}_+^e$ such that $\varphi(M) \neq 0$ and for fixed $e$ we have $a^n \varphi \in \mathcal{D}_+^e$ for all sufficiently large $n$. It follows that $\underline{M}_\mathcal{D} \neq 0$.
\end{proof}

\begin{Prop}
\label{TauIsFunctorial}
Let $R$ be a ring essentially of finite type over an $F$-finite field, $\mathfrak{a}$ an ideal, $t \in \mathbb{Q}_{\geq 0}$ and $\mathcal{C}$ a Cartier algebra. Denote by $\mathcal{A}$ the category of $F$-regular coherent non-degenerate $\mathcal{C}$-modules. Then $M \mapsto \tau(M, \mathcal{C}, \mathfrak{a}^t)$ defines a functor from $\mathcal{A}$ to the category of coherent $\mathcal{C}^{\mathfrak{a}^t}$-modules.
\end{Prop}
\begin{proof}
Let us denote the target category by $\mathcal{B}$. All test modules will be considered with respect to the Cartier algebra $\mathcal{C}^{\mathfrak{a}^t}$ and the operation $\underline{\phantom{M}}$ will always be taken with respect to $\mathcal{C}^{\mathfrak{a}^t}$ as well. By definition $\tau$ maps objects from $\mathcal{A}$ to $\mathcal{B}$. Let $\varphi: M \to N$ be a morphism in $\mathcal{A}$. We define a morphism $\tau(\varphi): \tau(M) \to \tau(N)$ by setting $\tau(\varphi)(m) = \varphi(m)$. Since $\varphi$ is compatible with $\mathcal{C}$ it is a fortiori compatible with the subalgebra $\mathcal{C}^{\mathfrak{a}^t}$. We still have to check that $\varphi(m) \in \tau(N)$. Note that this can be checked locally. We may assume that $\Supp M \cap \Supp N$ is non-empty. We want to apply \cite[Theorem 3.11]{blicklep-etestideale} and hence need to find a test element $a \in R$ that works simultaneously for $M$ and $N$.

Claim: There is $a \in \mathfrak{a}$ such that $D(a) \cap \Supp(\underline{M})$ is dense in $\Supp \underline{M}$ and such that $D(a) \cap \Supp(\underline{N})$ is dense in $\Supp \underline{N}$. This is purely topological and by Lemma \ref{SupportFPure} we may therefore replace $\underline{M}$ with $M$ and similarly for $\underline{N}$.

$D(a) \cap \Supp M$ is dense in $\Supp M$ if it has non-empty intersection with each irreducible component. This is the case if and only if $\mathfrak{a}$ is not contained in the union of the minimal primes of $M$. Assume to the contrary that $\mathfrak{a}$ is contained in the union of the minimal primes of $M$ and $N$. Then by prime avoidance $\mathfrak{a}$ is contained in some minimal prime of $M$ (say). But the minimal primes of $M$ are contained in $\Ass M$ which consists of $0$ and the set of zero-divisors on $M$. Since $\mathfrak{a}$ contains an $M$-regular element this is a contradiction proving the claim.

Since over $D(a)$ we have $\mathcal{C} \mathfrak{a}^{\lceil t p^e\rceil}[c^{-1}] = \mathcal{C}$ both $(M_a, \mathcal{C}^{\mathfrak{a}^t})$ and $(N_a,\mathcal{C}^{\mathfrak{a}^t})$ are $F$-regular. Hence, \cite[Theorem 3.11]{blicklep-etestideale} implies that $\tau(M)$ is generated as a $\mathcal{C}^{\mathfrak{a}^t}$-module by $a \underline{M}$ and that $\tau(N)$ is generated by $a \underline{N}$. Since $\underline{M} = (\mathcal{C}^{\mathfrak{a}^t}_+)^l M$ for $l \gg 0$ and $\underline{N} = (\mathcal{C}^{\mathfrak{a}^t}_+)^{l'} N$ for $l' \gg 0$ we may compute (choosing $\ell \geq l, l'$)
\[ \varphi(\tau(M)) = \varphi(\sum_{e \geq 1} \mathcal{C}_e^{\mathfrak{a}^t} a (\mathcal{C}^{\mathfrak{a}^t}_+)^\ell M) = \sum_{e \geq 1} \mathcal{C}_e^{\mathfrak{a}^t} a (\mathcal{C}^{\mathfrak{a}^t}_+)^\ell\varphi(M) \subseteq \sum_{e \geq 1} \mathcal{C}_e^{\mathfrak{a}^t} a (\mathcal{C}^{\mathfrak{a}^t}_+)^\ell N = \tau(N)\] which yields the claim.
\end{proof}

For later use we note the

\begin{Le}
\label{TauSurjectiveForNiliso}
Assume the situation in \ref{TauIsFunctorial} and that $\mathfrak{a} = (f)$ is principal. Let $\varphi: M \to N$ be a nil-isomorphism with $M, N \in \mathcal{A}$. Then the induced map $\tau(M) \to \tau(N)$ is surjective.
\end{Le}
\begin{proof}
Since $N$ is $F$-pure $\varphi$ is surjective. Surjectivity is local so we may assume that there is $c$ such that $\tau(M) = \sum_{e \geq 1} \mathcal{C}_e^{\mathfrak{a}^t} c (\mathcal{C}_{+}^{\mathfrak{a}^t})^\ell M$ and similarly for $\tau(N)$. Hence, for an element  $\kappa^e (rf^{\lceil tp^e\rceil} c n)$ in $\tau(N)$ we find $m$ in $(\mathcal{C}_{+}^{\mathfrak{a}^t})^\ell M$ with $\varphi(m) = n$. Since $\varphi$ is Cartier linear we obtain $\kappa^e( rf^{\lceil tp^e\rceil} c n) = \kappa^e( rf^{\lceil tp^e\rceil} c \varphi(m)) = \varphi(\kappa^e( rf^{\lceil tp^e\rceil} c m))$ and the claim follows.
\end{proof}

\begin{Def}
Let $M$ be a Cartier module with associated $V$-filtration. Then we denote the associated graded module $\bigoplus_{t \in \mathbb{Q}} V^{t-\eps}/V^{t}$ by $Gr^\cdot$. We also denote $V^{t - \eps}/V^t$ by $Gr^t$ or $Gr^t M$.
\end{Def}

Our next goal is to define, in the situation of Theorem \ref{VfiltrationTestidealfiltration}, a natural Cartier module structure on the $Gr^t$. Note that $\kappa f^{\lceil t(p-1)\rceil} \tau(M, \mathcal{C}^{\mathfrak{a}^t}) \subseteq \tau(M, \mathcal{C}^{\mathfrak{a}^t})$ if $(f) = \mathfrak{a}$.

\begin{Prop}
\label{GrSupportedonIdeal}
Let $R$ be a ring essentially of finite type over an $F$-finite field, $\mathfrak{a}$ an ideal and $\mathcal{C}$ a principal Cartier algebra generated by $\kappa$. Let $M$ be an $F$-regular coherent non-degenerate $\mathcal{C}$-module. Then the $Gr^t M$ are $R$-modules supported on $R/\mathfrak{a}$ for all $t \in \mathbb{Q}$. Moreover, if $\mathfrak{a} = (f)$ is principal then $Gr^t$ is a Cartier module with respect to $\kappa f^{\lceil t(p-1)\rceil}$.
\end{Prop}
\begin{proof}
By the Brian\c{c}on-Skoda theorem for test modules \cite[Theorem 4.21]{blicklep-etestideale} we have $\mathfrak{a} V^t \subseteq V^{t+1}$. Since $V^\cdot$ is decreasing, $V^{t +1} \subseteq V^{t - \eps}$. Thus $Gr^t M$ is supported on $R/\mathfrak{a}$.

Since both $V^t$ and $V^{t - \eps}$ are Cartier modules with respect to $\kappa f^{\lceil t (p-1)\rceil}$ we can form their quotient which is again a Cartier module.
\end{proof}

\begin{Bem}
A similar construction works for arbitrary ideals $\mathfrak{a}$. Indeed, set $\mathcal{C} = \langle\kappa \mathfrak{a}^{\lceil t \rceil (p-1)}\rangle$ and use the fact that $I^{[q]} \subseteq I^q$ for an ideal $I \subseteq R$.
\end{Bem}

\begin{Prop}
\label{CartierOperatesonGr}
Let $R$ be a ring essentially of finite type over an $F$-finite field, $\mathfrak{a}$ an ideal and $\mathcal{C}$ a principal Cartier algebra generated by $\kappa$. Let $M$ be a coherent Cartier module. Then $\kappa$ induces a surjective Cartier linear morphism $Gr^{tp} \to Gr^t$ for all $t \geq 0$.
\end{Prop}
\begin{proof}
Let $t \geq 0$ then by Lemma \ref{FiltrationsProp1} we have $\kappa(\tau(M, \kappa, \mathfrak{a}^{tp})) = \tau(M, \kappa, \mathfrak{a}^{t})$ and $\kappa(\tau(M, \kappa, \mathfrak{a}^{tp - \eps})) = \tau(M, \kappa, \mathfrak{a}^{t - \eps/p})$ so that $\kappa$ induces a surjective homomorphism of abelian groups $Gr^{tp} \to Gr^t$. Since the projections are $R$-linear and $\kappa$ is Cartier linear the induced morphism is also Cartier linear.
\end{proof}

\begin{Prop}
\label{MultWithFIsoOnGr}
Let $R$ be a ring essentially of finite type over an $F$-finite field, $(f) = \mathfrak{a}$ an ideal and $\mathcal{C}$ a principal Cartier algebra generated by $\kappa$. Let $M$ be an $F$-regular coherent non-degenerate $\mathcal{C}$-module. Then for all $t > 0$ we have an isomorphism $\mu_f: Gr^{t}M \to Gr^{t+1} M$ of Cartier modules induced by multiplication by $f$.
\end{Prop}
\begin{proof}
Note that multiplication by $f$ is injective on $M$.
By property (iii) of Definition \ref{DefVfiltration} we have $Gr^{t+1} M = \mathfrak{a} V^{t - \eps} / \mathfrak{a} V^t$. And multiplication with $f$ induces an isomorphism of $R$-modules. We still have to verify that the Cartier structures are preserved. We have $\mu_f( \kappa f^{\lceil t (p-1) \rceil }(m)) = f \kappa( f^{\lceil t  (p-1)\rceil } (m) = \kappa(f^p f^{\lceil t (p-1)\rceil } m) = \kappa(f^{\lceil (t +1)(p-1) \rceil} \mu_f(m))$ for any $m \in Gr^t(M)$.
\end{proof}

\begin{Bem}
We caution the reader that the construction of the Cartier structure on the associated graded involves the choice of a generator of $\mathfrak{a}$. So once one has fixed a generator of $\mathfrak{a}$ the functor $Gr$ induces a functor from $\mathcal{A}$ to the category of Cartier modules. If we choose a different generator $g$, then $Gr_f(M)$ and $Gr_g(M)$ will in general not be isomorphic as Cartier modules. Explicitly, take $\mathfrak{a} =(x)$ and $R = \mathbb{F}_p(t)[x]$ and $f = x, g = tx$. Then assuming $p \geq 3$ we have $Gr^1 R = R/(x)$. Assume that $Gr^1_f(R)$ and $Gr^1_g(R)$ are isomorphic. Since an isomorphism is necessarily $R$-linear it would have to be the identity. But this is not compatible with the Cartier structures since e.\,g.\ for $t^{p-1}$ we have $\kappa x^{p-1} (t^{p-1}) = 1$ while $\kappa t^{p-1} t^{p-1} = 0$.

However, since any two generators only differ by a unit they will always generate the same Cartier algebra and the obtained categories of Cartier modules are equivalent (the equivalence being induced by the obvious functor). Since we want to descend to crystals later on we only work with one single Cartier operator.
\end{Bem}

\begin{Theo}
\label{GrModuleFunctor}
Let $R$ be a ring essentially of finite type over an $F$-finite field, $\mathfrak{a}$ an ideal, $t \in \mathbb{Q}_{\geq 0}$ and $\mathcal{C}$ a principal Cartier algebra generated by $\kappa$. Denote by $\mathcal{A}$ the category of $F$-regular coherent non-degenerate $\mathcal{C}$-modules. Then $M \mapsto Gr^t M$ defines a functor from $\mathcal{A}$ to the category of coherent $\kappa^{\mathfrak{a}^t}$-modules. 
\end{Theo}
\begin{proof}
Let $\varphi: M \to N$ be a morphism. By Proposition \ref{TauIsFunctorial} this induces morphisms $\tau^t(\varphi): V^{t} M \to V^t N$ and $\tau^{t - \eps}(\varphi): V^{t - \eps} M \to V^{t - \eps} N$. Moreover, $\tau^{t}(\varphi)(V^{t} M) = \tau^{t - \eps}(\varphi)(V^{t} M)$. Hence, we obtain a morphism of abelian groups $Gr^t(\varphi): Gr^t M \to Gr^t N$. Since $\varphi$ is a morphism of Cartier modules and the Cartier structures on $Gr^t$ are induced by those on $M, N$ it follows that $Gr^t(\varphi)$ is a morphism of Cartier modules.
\end{proof}

\begin{Bem}
It is important to note that we will use in many cases the full machinery of test modules in order to prove our statements. They do not simply follow from the axioms of the $V$-filtration. Also note that with the present definition of test modules the association $M \mapsto \tau(M)$ is \emph{not} functorial in general. For instance, in \cite[Remark 3.3]{blicklep-etestideale} an example is given, where an inclusion $M \to N$ of modules does not restrict to an inclusion $\tau(M) \to \tau(N)$ of test modules. Hence, we need some extra conditions, like those in Proposition \ref{TauIsFunctorial}, to ensure functoriality.
\end{Bem}

\section{Crystals}
\label{SectionCrystals}
In order to have an equivalence of categories with perverse constructible sheaves on $X_{\acute{e}t}$ we need to pass from $\kappa$-modules to crystals (after an embedding into a separated smooth scheme). The purpose of this section is to show that the construction of the functor $Gr$ also makes sense if we pass to crystals. Also note that in the case of a crystal the Cartier structure is always given by one single morphism. We will stress this by using the phrase $\kappa$-module rather than Cartier module.

The next result is immediate from \cite[Proposition 3.2 (e), Lemma 3.5]{blicklep-etestideale}.

\begin{Le}
\label{LeFRegularLocalizes}
$F$-regularity localizes in the sense that if $M$ is an $F$-regular coherent Cartier module then $S^{-1}M$ is $F$-regular for any multiplicative system $S \subseteq R$. 
\end{Le}

\begin{Prop}
\label{FregularCharacterization}
Let $M$ be a coherent $\mathcal{C}$-module and assume that $\tau(M, \mathcal{C})$ exists.
The following are equivalent:
\begin{enumerate}[(i)]
 \item{$M$ is $F$-regular.}
\item{For all $c \in R$ such that $\Supp M \cap \Spec R_c$ lies dense in $\Supp M$ we have $\sum_{e \geq 1} \mathcal{C}_e cM = M$.}
\end{enumerate}
\end{Prop}
\begin{proof}
This is essentially a reformulation of \cite[Theorem 3.11]{blicklep-etestideale}.
\end{proof}

\begin{Def}
A coherent $\kappa$-crystal $M$ is called \emph{$F$-regular} if there is an $F$-regular $\kappa$-module $N$ whose associated crystal is equivalent to $M$. 
\end{Def}

\begin{Bsp}
Note that not every $\kappa$-module associated to an $F$-regular $\kappa$-crystal is necessarily $F$-regular. Indeed, let $k$ be a perfect field and let $V$ be a $\kappa$-module. Then $V$ has a decomposition $V = V_{\nil} \oplus \underline{V}$ and $\underline{V}$ has no non-trivial nilpotent subspaces. If $V_{\nil} \neq 0$ then $V$ is not $F$-regular. But any $F$-pure $\kappa$-module over $k$ will be $F$-regular -- this follows from the criterion established in Proposition \ref{FregularCharacterization}. So $\underline{V}$ is $F$-regular and hence the crystal associated to $V$ is also $F$-regular (which is in fact isomorphic to $\underline{V}$ -- cf.\ \cite[Theorems 3.10, 3.12]{blickleboecklecartierfiniteness}).
\end{Bsp}
 
\begin{Le}
\label{FRegularReductionCrystal}
Let $M$ be a coherent $F$-pure $\kappa$-module. Then its associated minimal $\kappa$-module $M_{\min}$ (as defined in Section \ref{SectionCartierModules}) is $F$-regular if and only if $M$ is $F$-regular.
\end{Le}
\begin{proof}
Let $N$ be the maximal nilpotent $\kappa$-submodule of $M$. Then $M_{\min} = M/N$ and $\Supp M_{\min} = \Supp M$ by \cite[Lemma 3.16]{blickleboecklecartierfiniteness}.

First assume that $M$ is $F$-regular.
We use Proposition \ref{FregularCharacterization} to show that $M_{\min}$ is $F$-regular. For $c$ such that $D(c) \cap \Supp M_{\min}$ is dense in $\Supp M_{\min}$ we obtain $\sum_{e \geq 1} \kappa^e(cM/N) = \sum_{e \geq 1} \kappa^e(cM) + N = M +N = M/N$ since $M$ was $F$-regular.

If $M_{\min}$ is $F$-regular then for every $c$ as above we have $\mathcal{C}_+ c M_{\min} = M_{\min}$. We have to show that also $\mathcal{C}_+ c M = M$. Given $m \in M$ there is  $e  \geq 1$ and $m' \in M$ such that $\kappa^e(cm') - m \in N$. In particular, there is $s \geq 0$ (independent of $m, m'$) such that $\kappa^{e+s}(c m') = \kappa^s(m)$. Since $M$ is $F$-pure $\kappa$ is surjective. Thus varying $m$ we get that every element of $M$ is of the form $\kappa^{e+s}(c m')$ for some $m' \in M$ and some $e \geq 1$
\end{proof}

\begin{Ko}
A $\kappa$-crystal $M$ is $F$-regular if and only if for every subcrystal $N$ of $M$ which agrees with $M$ for every generic point $\eta$ of $\Supp M$ (cf. Section \ref{SectionCartierModules} for the definition of support of a crystal) we have $N = M$.
\end{Ko}
\begin{proof}
Assume that $M$ is $F$-regular. By Lemma \ref{FRegularReductionCrystal} the $\kappa$-module $M_{\min}$ is then $F$-regular. Let now $N \to M$ be a subcrystal of $M$ generically agreeing with $M$. Since the categories of minimal $\kappa$-modules and $\kappa$-crystals are equivalent (\cite[Theorem 3.12]{blickleboecklecartierfiniteness}) we obtain a morphism $N_{\min} \to M_{\min}$ of minimal $\kappa$-modules. Moreover, the kernel of this morphism is nilpotent (since it is a monomorphism after passing to crystals -- cf. \cite[Proposition 2.3.5 (a)(iii)]{boecklepinktaucrystals}). But then it must be injective since $N_{\min}$ by definition does not admit nilpotent submodules.

In the other direction it suffices to show that $M_{\min}$ is $F$-regular. So let $N \subseteq M_{\min}$ be a $\kappa$-submodule which generically agrees with $M_{\min}$. Since $\underline{\phantom{M}}$ commutes with localization (cf. \cite[Lemma 2.18]{blicklep-etestideale}) we have $\underline{N}_{\eta} = (M_{\min})_\eta$ for every generic point $\eta$ of $\Supp(M_{\min})$. In particular, $\underline{N}$ is an $F$-pure submodule generically agreeing with $M_{\min}$. Finally, $\underline{N} = \overline{\underline{N}}$ since any nilpotent submodule of $\underline{N}$ is a nilpotent submodule of $M_{\min}$ hence zero. Hence, $\underline{N}$ yields, via the equivalence between minimal $\kappa$-modules and $\kappa$-crystals, a subcrystal $N$ of $M$ that generically agrees with $M$. By assumption $M = N$ and we conclude that $M_{\min}$ is $F$-regular.
\end{proof}

Since in the category of $\kappa$-crystals the structural maps are isomorphisms the notion of $F$-purity is vacuous for crystals. Hence, the above corollary shows that $F$-regularity for crystals really is the analog of $F$-regularity for modules.

\begin{Le}
Let $M$ be a coherent $\kappa$-module non-degenerate with respect to some ideal $\mathfrak{a}$. Then its associated minimal $\kappa$-module $M_{\min}$ is also non-degenerate.
\end{Le}
\begin{proof}
Let $N$ be the maximal nilpotent $\kappa$-submodule of $M$ and consider the quotient $\overline{M} = M/N$. Let $x \in \mathfrak{a}$ be a regular element for $M$. We have to show that $x$ is regular for $\overline{M}$. So fix $m \in M$ and assume that $x m \in N$. This implies that the $\kappa$-module generated by $xm$ is nilpotent. This means that there is $e > 0$ such that for all $f \in R$ and $n \geq 0$ we have $\kappa^e(\kappa^n(fxm)) = 0$. In particular, taking $f = x^{p^{e+n} - 1} r$ for $r \in R$ we obtain $\kappa^{e+n}(x^{p^{e+n}} rm) = x \kappa^{e+m}(rm) = 0$. This implies $\kappa^{e+m}(rm) = 0$ so that already the $\kappa$-module generated by $m$ is nilpotent.

Since $\underline{M}$ is a submodule of $M$ any regular element $x \in \mathfrak{a}$ is a fortiori regular in $\underline{M}$. 
\end{proof}

In the next theorem we want to show that the functor $Gr$ defined in \ref{GrModuleFunctor} descends to crystals. In order to make sense of this we need to assume that the Cartier structures on the source and target are given by a single Cartier linear morphism. Hence, we can only carry out this construction if $\mathfrak{a}$ is a principal ideal.

\begin{Theo}
\label{GrFunctoronCrys}
Let $R$ be a ring essentially of finite type over an $F$-finite field, $\mathfrak{a} = (f)$ a principal ideal and $t \in \mathbb{Q}_{\geq 0}$.
Then the functor $Gr^t$, for $t \in \mathbb{Q}_{\geq 0}$ defined in \ref{GrModuleFunctor} induces a functor from the category of $F$-regular coherent non-degenerate $\kappa$-crystals on $R$ to the category of coherent $\kappa f^{\lceil t (p-1)\rceil}$-crystals on $R/\mathfrak{a}$.
\end{Theo}
\begin{proof}
We consider $\mathcal{A}$ (as defined in \ref{GrModuleFunctor}) as a (full) subcategory of the category of coherent $\kappa$-modules. It is then clear that nil-isomorphisms in $\mathcal{A}$ form an essentially locally small multiplicative system in the sense of \cite[Definitions 2.2.1, 2.2.2]{boecklepinktaucrystals}. Hence, in order to show that $Gr^t$ induces a functor on crystals it is sufficient to show that nil-isomorphisms are mapped to nil-isomorphisms \cite[Proposition 2.2.4]{boecklepinktaucrystals}. So let $M, N$ be in $\mathcal{A}$ and let $\varphi: M \to N$ be a nil-isomorphism.

Consider the diagram
\[
 \begin{xy} \xymatrix{\tau(M, \kappa_M, \mathfrak{a}^{t-\eps}) \ar[d] \ar[r]^{\tau(\varphi)} &\tau(N,\kappa_N, \mathfrak{a}^{t-\eps}) \ar[d] \\
Gr^t(M) \ar[r]^{Gr(\varphi)} & Gr^t(N)}
\end{xy}
\] where $Gr^t(\cdot) = \tau(\cdot, \kappa_{\cdot}, \mathfrak{a}^{t-\eps})/\tau(\cdot, \kappa_{\cdot}, \mathfrak{a}^t)$, $\tau(\varphi)$ and $Gr(\varphi)$ are restrictions of $\varphi$ and the vertical arrows are the canonical projections. By Lemma \ref{TauSurjectiveForNiliso} the map $\tau(\varphi)$ is surjective. In particular, $Gr^t(\varphi)$ is surjective. The kernel of $Gr^t(\varphi)$ is \[(\ker \varphi \cap \tau(M, \kappa_M, \mathfrak{a}^{t - \eps}) + \varphi^{-1}(\tau(N, \kappa_N, \mathfrak{a}^t))/ \tau(M, \kappa_M, \mathfrak{a}^t),\] and this is equal to $(\ker \varphi \cap \tau(M, \kappa_M, \mathfrak{a}^{t-\eps}))/\tau(M, \kappa_M, \mathfrak{a}^t)$. Since $\ker \varphi$ is a nilpotent $\kappa$-module $\ker(Gr^t(\varphi))$ is nilpotent as a $\kappa f^{\lceil t (p-1)\rceil}$-module. 
\end{proof}

\begin{Ko}
In the situation of \ref{GrFunctoronCrys} multiplication by $f$ induces an isomorphism of crystals $\mu_f: Gr^t \to Gr^{t+1}$ for $t > 0$. Moreover, $\kappa$ yields a surjective Cartier linear map $Gr^{tp} \to Gr^t$ for $t \geq 0$.
\end{Ko}
\begin{proof}
This follows from Propositions \ref{CartierOperatesonGr} and \ref{MultWithFIsoOnGr}.
\end{proof}

\begin{Bem}
\label{WrongProp}
In a previous version of this article it was claimed that $F$-regularity is preserved for open immersions provided one passes to crystals and uses the notion of local nilpotence. This is however false, as the following example illustrates.

First of all, let us recall the set-up. Let $f: \Spec S \to \Spec R$ be of finite type and $M$ a Cartier module. Then $f_\ast M$ is of course in general not coherent. However, if one is working in the category of quasi-coherent crystals (that is quasi-coherent $\kappa$-modules up to local nilpotence -- cf.\ \cite[Definition 2.1.1]{blickleboecklecartiercrystals}) then the crystal associated to $f_\ast M$ is isomorphic to a \emph{coherent} $\kappa$-crystal. This allows one to define a functor $f_\ast$ on coherent crystals for $f$ of finite type (cf.\ \cite[Theorems 3.2.14, 3.2.18]{blickleboecklecartiercrystals}).

Let now $S = k[x,x^{-1}], R = k[x]$ for a perfect field $k$ and let $f: \Spec S \to \Spec R$ be the open immersion and $M = k[x,x^{-1}]$ with Cartier structure given by the Cartier operator acting via \[\kappa^e\left(\frac{g}{x^n}\right) = \frac{\kappa^e\left(gx^{n(p^e-1)}\right)}{x^n},\] which is just the usual action obtained on the localization. In particular, if $g$ is a monomial of degree $m$ then the degree of $\kappa^e(g/x^n)$ is given by \[\frac{m +(n-1)(p^e-1)}{p^e}  -n,\] if the numerator is divisible by $p^e$ and otherwise the element is mapped to zero.

We claim that the inclusion $\varphi: k[x] \cdot x^{-1} \to f_\ast M$ is a nilisomorphism. For this we have to show that the quotient $Q := f_\ast M/k[x] \cdot x^{-1}$ is locally nilpotent. By definition, we must verify that $Q$ is the union of nilpotent $\kappa$-modules. Consider $Q_n = k[x] \cdot x^{-n} \subseteq Q$ for $n \geq 2$. Any element in $Q_n$ is of the form $g/x^n$ with $\deg g \leq n-2$ so by the above degree formula (looking at monomials separately) there is $e_n$ such that $\kappa^{e_n} Q_n = 0$.

Finally, $N := k[x] \cdot x^{-1}$ admits the Cartier submodule $k[x]$ which generically agrees with $N$ and therefore $N$ is not $F$-regular. On the other hand, $k[x,x^{-1}]$ is $F$-regular by Lemma \ref{RegularimpliesFRegular} below. It is also easy to see that both $N$ and $M$ are minimal $\kappa$-modules. Hence, $F$-regularity is not preserved by open immersions on the level of crystals.

\end{Bem}

\section{Functorial properties}
\label{SectionFunctorialProperties}

In this section we show that the functor $Gr$ on $\kappa$-modules (or crystals in the principal case) as defined in \ref{GrModuleFunctor} (resp.\ \ref{GrFunctoronCrys}) commutes with pushforwards along finite morphisms and with pull backs along essentially \'etale morphisms. This is accomplished by first proving the desired property for test modules (or $V$-filtrations) and then passing to quotients. In particular, we derive corresponding transformation rules for test modules.

In dealing with pushforwards we will consider (affine) schemes $X = \Spec R$ over $\mathbb{A}^1_{\mathbb{F}_p} = \Spec \mathbb{F}_p[x]$ and $V$-filtrations on $X$ along $\mathfrak{a} = (a)$, where $a$ is the image of $x$ under the structural morphism. Abusing notation we will frequently denote $a$ by $x$.

We recall the following lemma from commutative algebra:

\begin{Le}
\label{FiniteMorphismSupport}
Let $f: \Spec S \to \Spec R$ be a finite dominant morphism with $R, S$ noetherian, reduced and $M$ a coherent $S$-module. Then $f(\Supp M) = \Supp f_\ast M$.
\end{Le}

\begin{Le}
\label{VFiltrationFinitePushforward}
Let $f: \Spec S \to \Spec R$ be a finite morphism of schemes over $\mathbb{A}^1_{\mathbb{F}_p}$. Let $M$ be a principal Cartier module on $S$ whose $V$-filtration along $\mathfrak{a} = (x)$ exists. Then $f_\ast M$ has a $V$-filtration along $(x)$ and $f_\ast V^\cdot M = V^\cdot f_\ast M$.
\end{Le}
\begin{proof}
We verify that $f_\ast V^\cdot M$ satisfies the axioms of a $V$-filtration for $f_\ast M$. Obviously the filtration is $\mathbb{Q}$-indexed, exhaustive, decreasing and right-continuous. Since $f$ is finite $f_\ast V^0 M$ is coherent as an $R$-module. Continuity in zero is clear.

By axiom (ii) for $V^\cdot M$ we find $s x \in \mathfrak{a} \subseteq S$ such that $\mu_{(sx)^j}: V^t \to V^{t+j}$ is injective. Since $S$ is commutative $\mu_{(sx)^j}$ factors as $\mu_{s^j} \circ \mu_{x^j}$, where $\mu_{s^j}: V^t \to V^t$ and $\mu_{x^j}: V^t \to V^{t+j}$ by axiom (iii). Hence, already $\mu_{x^j}$ is injective.

Next, we verify that axiom (iii) is satisfied. Fix $m \in f_\ast V^t$. Then we find $s \in S$ and $m' \in f_\ast V^{t-n}$ such that $s x^n m' = m$. Now set $m'' := s m'$, then $x^n m'' = m$ as desired. The other inclusion is clear.

Finally, axiom (iv) follows since $F_\ast \circ f_\ast$ and $f_\ast \circ F_\ast$ are naturally isomorphic.
\end{proof}

\begin{Le}
\label{FRegularPushforward}
Let $R$ be essentially of finite type over an $F$-finite field. Let $f: \Spec S \to \Spec R$ be a finite morphism and $M$ an $F$-regular coherent $\kappa$-module over $S$. Then $f_\ast M$ is also $F$-regular.
\end{Le}
\begin{proof}
We may factor $f$ as $i \circ f'$, where $f'$ is finite dominant and $i$ is a closed immersion. Since pushforward along closed immersions clearly preserves $F$-regularity (cf.\ \cite[Discussion before Lemma 2.20]{blicklep-etestideale}) we restrict to the case that $f$ is finite dominant.
Denote $\Ann_S(M)$ by $I$. By Lemma \ref{FiniteMorphismSupport} and since $f_\ast M$ is clearly $F$-pure the whole situation factors through $\Spec S/I \to \Spec R/(I \cap R)$. Hence, we may assume that $\Supp M = \Spec S$, where $f$ is still finite dominant and that $R, S$ are reduced. Using Proposition \ref{FregularCharacterization} it remains to verify that $\mathcal{C}_+ c f_\ast M = f_\ast M$ for all $c \in R^{\circ}$. Since we are assuming that $M$ is $F$-regular this boils down to the assertion that if $c \in R^{\circ}$ then $c \in S^{\circ}$ which follows from \cite[Corollary 4.18]{eisenbud} (here we use reducedness).
\end{proof}

\begin{Prop}
Let $R$ be essentially of finite type over an $F$-finite field. Let $f: \Spec S \to \Spec R$ be a finite morphism of schemes over $\mathbb{A}^1_{\mathbb{F}_p}$, $\mathfrak{a} = (x)$ an ideal in $S$ and $M$ an $F$-regular coherent non-degenerate $\kappa$-module over $S$. Then \[f_\ast \tau(M, \kappa_M^{(x)^t}) = \tau(f_\ast M, \kappa_{f_\ast M}^{(x)^t})\] for any $t \in \mathbb{Q}_{\geq 0}$.
\end{Prop}
\begin{proof}
Note that if $\mathfrak{a}$ contains an $M$-regular element then $x$ is $M$-regular. In particular, $x$ is also $f_\ast M$-regular. Now the  claim follows from Theorem \ref{VfiltrationTestidealfiltration} and Lemmata \ref{VFiltrationFinitePushforward} and \ref{FRegularPushforward}.
\end{proof}

\begin{Ko}
Let $f: \Spec S \to \Spec R$ be a finite morphism of schemes over $\mathbb{A}^1_{\mathbb{F}_p}$. Let $M$ be a principal Cartier module on $S$ whose $V$-filtration along $(x)$ exists. Then $f_\ast Gr M = Gr f_\ast M$ in the category of R-modules.

If in addition $R$ is essentially of finite type over an $F$-finite field and $M$ is coherent and $F$-regular so that $Gr M$ is endowed with a $\kappa$-module structure then $f_\ast Gr M = Gr f_\ast M$ also for $\kappa$-modules or crystals.
\end{Ko}
\begin{proof}
The functor $f_\ast$ is exact since $f$ is affine. Hence, $f_\ast V^{t- \eps} M/ f_\ast V^t M = f_\ast Gr^t M$ and using Lemma \ref{VFiltrationFinitePushforward} we obtain $f_\ast Gr^t M = Gr^t f_\ast M$ as $R$-modules. This isomorphism intertwines the Cartier structures so that it induces an isomorphism of Cartier modules and also an isomorphism of crystals.
\end{proof}

Given a morphism $f: \Spec S \to \Spec R$ over $\mathbb{A}^1_{\mathbb{F}_p}$ we obtain a commutative diagram
\[
\begin{xy}
 \xymatrix{ \Spec S \ar[r]^f & \Spec R \\
\Spec S/(x) \ar[r]^{\bar{f}} \ar[u]^{i_S} & \Spec R/(x) \ar[u]^{i_R}}
\end{xy}
\]
of morphisms over $\mathbb{A}^1_{\mathbb{F}_p}$, where the vertical arrows denote closed immersions and the bottom horizontal arrow is obtained by restricting $f$. The functor $f_\ast$ on module categories induces a functor of $\kappa$-modules which also descends to crystals (cf.\ \cite[Lemma 2.2]{blickleboecklecartierfiniteness} and \cite[Section 3.2]{blickleboecklecartiercrystals}). For a ring $R$ denote by $\mathcal{A}_{R}$ the category of $F$-regular coherent non-degenerate $\mathcal{C}$-modules.

\begin{Theo}
\label{GrPushforwardCommute}
Let $R$ be essentially of finite type over an $F$-finite field. Let $f: \Spec S \to \Spec R$ be a finite morphism of schemes over $\mathbb{A}^1_{\mathbb{F}_p}$. Then one has a natural isomorphism of functors $\bar{f}_\ast \circ Gr_S = Gr_R \circ f_\ast$, where $Gr_R$ maps from $\mathcal{A}_{R}$ to the category of $\kappa$-modules over $R/(x)$ and similarly for $Gr_S$. Moreover, we still have an isomorphism of functors if we pass to crystals on the source and target.
\end{Theo}
\begin{proof}
Clearly, the identity induces an isomorphism of functors between $\bar{f}_\ast \circ Gr_S$ $Gr_R \circ f_\ast$ on the level of $\kappa$-modules. We have seen in Theorem \ref{GrFunctoronCrys} above that the functors $Gr_S$ and $Gr_R$ descend to crystals. Since any isomorphism of Cartier modules is a nil-isomorphism the isomorphism of functors descends to crystals. 
\end{proof}

\begin{Bem}
\begin{enumerate}[(a)]

\item{If we consider $Gr_{\bullet}$ as a functor to $\kappa$-modules on $\Spec \bullet$, where $\bullet = R, S$ then we have for any module $M$ in the image of $Gr_S$ or $Gr_R$ that $(x) \subseteq \Ann(M)$. Hence, using \cite[Propositions 2.6]{blickleboecklecartierfiniteness} one also obtains the formula $f_\ast \circ Gr_S = Gr_R \circ f_\ast$ using this alternative definition of $Gr$.

For crystals we have an equivalence of categories between $\kappa$-crystals on $\Spec R/(x)$ and $\kappa$-crystals on $\Spec R$ that are supported in $V(x)$ (cf.\ \cite[Theorem 4.1.2]{blickleboecklecartiercrystals}, \cite[Propositions 3.19]{blickleboecklecartierfiniteness}). Hence, we also obtain the formula $f_\ast \circ Gr_S = Gr_R \circ f_\ast$ if we consider $Gr_{\bullet}$ as a functor to $\kappa$-crystals on $\Spec \bullet$, where $\bullet = R, S$.}
\item{If $\mathfrak{a}$ is not principal then it is not clear, whether axioms (ii) and (iii) hold for $f_\ast V$. Also note that one can probably see the equality $\tau(f_\ast M) = f_\ast \tau(M)$ directly for principal $\mathfrak{a}$ first by using reductions similar to \cite[Proposition 4.16]{blicklep-etestideale} and then employing arguments similar to those used in the proof of Theorem \ref{FRegularEtalePullback}.} 
\item{The reader should compare Theorem \ref{GrPushforwardCommute} to \cite[Expos\'e XIII, 1.3.8]{SGA7II}, where it is shown that for $f$ finite and $\Psi$ the nearby cycles functor the natural transformation of functors $f_! \Psi \to \Psi f_!$ is an isomorphism.
}
\end{enumerate}
\end{Bem}

We now turn our attention to $f^!$ for $f$ essentially \'etale. Recall, that for $f: X \to Y$ essentially \'etale one has $f^! = f^\ast$ (cf.\ \cite[Part I, Theorem 4.8.1]{lipmangsduality}).
The following lemma recalls how the Cartier structure on $f^! M$ is obtained in this case.

\begin{Le}
\label{CartierStructureEtalePullback}
Let $f: X = \Spec S \to Y = \Spec R$ be a essentially \'etale morphism of affine schemes and let $M$ be an $R$-module and denote its base change to $S$ by $M'$. Then any $R$-linear morphism $F_\ast M \to M$ lifts uniquely to an $S$-linear morphism $F_\ast(M') \to M'$.
\end{Le}
\begin{proof}
One has an isomorphism $\varphi: F_\ast M \otimes_R S \to F_\ast(M \otimes_R S)$ given by $m \otimes s \mapsto m \otimes s^p$ (cf.\ \cite[Lemma 2.2.1 and Definition 2.2.2]{blickleboecklecartiercrystals}). The claim now readily follows. 
\end{proof}

\begin{Le}
\label{NilAdjInjective}
Let $R$ be a noetherian $F$-finite ring and $M$ a $\kappa$-module. Then $M_{\nil} = 0$ if and only if the natural map $M \to F^! M$ is injective.
\end{Le}
\begin{proof}
The natural map $M \to F^! M = \Hom_R(F_\ast R, M)$\footnote{To be precise we should write $H^0 F^!$ here.} is induced by adjunction from $F_\ast M \to M$. Explicitly it is given by $m \mapsto (r \mapsto \kappa(rm))$. An element $m \in M$ is in the kernel of this map if and only if $\kappa(rm) = 0$ for all $r \in R$. Thus if $m \neq 0$ is contained in the kernel then the $R$-module generated by $m$ defines a nilpotent Cartier-submodule.

Conversely, assume that there is $e > 0$ such that $\kappa^e(rm) =0$ for some non-zero $m \in M$ and all $r \in R$. Fix the minimal $1 \leq s \leq e$ such that $\kappa^{e-s}(r'm)$ is non-zero for some $r' \in R$. Then we obtain $0 = \kappa^{e-s+1}(r^{p^{e-s}} r' m)) = \kappa(r \kappa^{e -s}(r'm)$ for all $r \in R$. Hence, $m' = \kappa^{e-s}(r'm)$ is a non-trivial element of the kernel of $M \to F^! M$.
\end{proof}

\begin{Le}
\label{LNilAndEtaleShriekCommute}
Let $\Spec S \to \Spec R$ be \'etale and essentially of finite type (or finite flat) where $R$ is noetherian and $F$-finite. Then $f^!(M_{\nil}) = (f^! M)_{\nil}$ for any coherent $\kappa$-module $M$.
\end{Le}
\begin{proof}
Since $f^!$ is exact we obtain the exact sequence $0 \to f^!M_{\nil} \to f^! M \to f^! \overline{M} \to 0$ from the exact sequence $0 \to M_{\nil} \to M \to \overline{M} \to 0$. An easy computation shows that $f^! M_{\nil}$ is nilpotent. We still need to show that $f^! \overline{M}$ does not admit non-trivial nilpotent submodules. This is the case if and only if $0 \to f^! \overline{M} \to F^! f^! \overline{M}$ is exact by Lemma \ref{NilAdjInjective} above. By assumption $0 \to \overline{M} \to F^! \overline{M}$ is exact. Applying $f^!$ (which is exact) and taking into account that $f^! F^! \overline{M} \cong F^! f^! \overline{M}$ yields the claim. 
\end{proof}

\begin{Bem}
\label{BBErrata}
Note that Lemma \ref{LNilAndEtaleShriekCommute} is essentially a generalization of \cite[Lemma 3.9]{blickleboecklecartierfiniteness}. Also note that the proof of this Lemma contains a gap. The statement of Lemma \ref{NilAdjInjective} as well as the fact that one can use it to close this gap was communicated to us by Manuel Blickle. It was also pointed out to us by Manuel Blickle that part (a) of \cite[Proposition 2.27]{blickleboecklecartierfiniteness} requires $k$ perfect while part (b) only holds under the additional assumption that $V = \overline{(\underline{V})}$.
\end{Bem}

\begin{Le}
\label{FPureEtalePullback}
Let $f: \Spec S \to \Spec R$ be an essentially \'etale morphism of affine schemes and let $M$ be a coherent $R$-Cartier module. If $M$ is $F$-pure then so is $M' = f^! M$. The converse holds if $f$ is surjective.
\end{Le}
\begin{proof}
We have to verify that $\mathcal{C}'_+ M' = M'$, where $\mathcal{C}'$ is the Cartier algebra obtained from $\mathcal{C}$ via Lemma \ref{CartierStructureEtalePullback}. This readily follows using the isomorphism $\varphi$ described in Lemma \ref{CartierStructureEtalePullback}. If $f$ is surjective then $f^!$ is faithfully flat so that $\mathcal{C}'_+ M' = \mathcal{C}_+ M \otimes_R S = M \otimes_R S$ implies $\mathcal{C}_+ M = M$.
\end{proof}

\begin{Le}
\label{FRegularConnComponents}
Let $R$ be a noetherian ring and $M$ a coherent Cartier module. Assume that $R$ is a finite direct product of rings $R_1, \ldots, R_n$ with inclusions $i_j: \Spec R_j \to \Spec R$. Then $M$  is $F$-regular if and only if each $i_j^! M$ is $F$-regular.
\end{Le}
\begin{proof}
Given Cartier modules $(M_i, \mathcal{C}_i)$ on $\Spec R_i$ for each $i$ we obtain a Cartier module on $\Spec R$ via $(\bigoplus_{i} M_i, \mathcal{C})$, where $\mathcal{C}$ acts on $M_i$ as $\mathcal{C}_i$. We claim that all Cartier modules on $\Spec R$ are of this form. This is clear for the module structure. So we just have to verify that every morphism $\kappa \in \Hom(F^e_\ast M, M)$ arises in this way. But $\Hom(F_\ast^e M, M) = \bigoplus_{i,j} \Hom(F_\ast^e M_i, M_j) = \bigoplus_i \Hom(F_\ast^e M_i, M_i)$ since for $i \neq j$ the supports of $F_\ast^e M_i$ and $M_j$ are disjoint. In particular, $M$ is $F$-pure if and only if all $i_j^! M$ are $F$-pure.

Assume now that all $i_j^! M$ are $F$-regular. If $N \subseteq M$ is a Cartier-submodule that generically agrees with $M$ then $i^!_j N$ is a Cartier submodule that generically agrees with $M_j$ for each $j$. Hence, $i^!_j N = M_j$ and therefore $N = M$. Conversely, if $M_{j_1}$ admits a proper Cartier-submodule $N_{j_1}$ that generically agrees with $M_{j_1}$ then (after possibly reordering so that $j_1 = 1$) $N = N_1 \oplus M_2 \cdots \oplus M_n$ is a proper Cartier-submodule that generically agrees with $M$.
\end{proof}

The following Lemma is well-known, see e.\,g.\ \cite[Theorem on p.\ 87]{hochsterlecturenotes}.

\begin{Le}
\label{RegularimpliesFRegular}
Let $R$ be an $F$-finite regular ring. Denote by $\mathcal{C}$ the algebra generated by $\Hom_R(F_\ast R, R)$. Then $R$ considered as a Cartier module via $\mathcal{C}$ is $F$-regular.
\end{Le}

\begin{Theo}
\label{FRegularEtalePullback}
Let $f: Y =\Spec S \to X =\Spec R$ be an essentially \'etale morphism of $F$-finite noetherian rings, $\mathfrak{a} \subseteq R$ an ideal and $t \in \mathbb{Q}_{\geq 0}$. Let $M$ be a coherent non-degenerate $\kappa^{\mathfrak{a}^t}$-module. If $M$ is $F$-regular then $f^! M$ is $F$-regular. If $f$ is also surjective the converse holds.
\end{Theo}
\begin{proof}
First of all, we reduce to the case $\Supp M = \Spec R$. Denote $\Ann_R(M)$ by $I$. We then have $V(I) = \Supp M$ and by base change we obtain an \'etale morphism $Y \times_X V(I) \to V(I)$. Moreover, $Y \times_X V(I)$ is isomorphic to $S/IS$ and $\Supp f^! M = \Spec S/IS$ the equality being due to \cite[Chap.\ 0, Corollaire 5.2.4.1]{EGAISpringer}. In view of \cite[Proposition 2.21]{blicklep-etestideale} this does not affect $F$-regularity.  Note that by Lemma \ref{FPureEtalePullback} we may assume that $M$ and $f^!M$ are $F$-pure. Hence the scheme structures on $\Spec R$ and $\Spec S$ are reduced.

Assume now that $f$ is surjective and assume that $M$ is not $F$-regular (but $F$-pure) and denote by $\mathcal{C}'$ the base change of $\mathcal{C}$. Then there is $c \in R^{\circ}$ such that $\mathcal{C}_+ c M \neq M$. Note that $R^\circ \subseteq S^{\circ}$ and $\mathcal{C}'_+ c f^! M = \mathcal{C}_+ c M \otimes_R S \neq f^! M$ by faithful flatness. Hence, $f^! M$ is also not $F$-regular showing the supplement of the theorem.

We now turn to the case that $M$ is $F$-regular and $f$ is again only essentially \'etale. We claim that it is enough to find $c \in R^{\circ}$ such that $(f^! M)_c$ is $F$-regular. Indeed, assuming this we obtain from \cite[Theorem 3.11]{blicklep-etestideale} that $\tau(f^! M) = \mathcal{C}'_+ cf^! M$ (note that $R^\circ \subseteq S^\circ$ by flatness). But $\mathcal{C}'_+ cf^! M$ equals $f^! M$ since $c$ is in $R$ and $M$ is $F$-regular. In particular, $\tau(f^! M) = f^! M$ so that $f^! M$ is $F$-regular.

Next, we want to reduce to the case that $\Supp M = \Spec R$ is integral. Since $R$ is reduced and excellent ($F$-finite implies excellent -- cf.\ \cite{kunznoetherian}), there is a dense open subset on which $R$ is regular. Suitably localizing (i.\,e.\ passing to $R_c$ for some $c \in R^{\circ}$ such that $R_c$ is regular) we may therefore assume that $R$ is regular. That is, $R$ is a finite direct product of excellent regular domains. Since $f$ is \'etale, $S$ is also a finite product of excellent regular domains. Applying Lemma \ref{FRegularConnComponents} to $S$ we may therefore assume that $\Spec S \to \Spec R$ is an essentially \'etale morphism of regular domains. Further replacing $c$ by $ca$ for some $a \in \mathfrak{a} \cap R^{\circ}$ we may assume that $\kappa^{\mathfrak{a}^t}[c^{-1}]$ is generated by $\kappa$. 
Finally, we may assume that $f^! M = (f^! M)_{\min}$ by Lemma \ref{FRegularReductionCrystal} (again $f^! M$ is $F$-pure by Lemma \ref{FPureEtalePullback}).

Using Lemma \ref{LNilAndEtaleShriekCommute} we may apply \cite[Proposition 2.27 (b)]{blickleboecklecartierfiniteness} (cf.\ also Remark \ref{BBErrata} above) and obtain that there is a finite separable extension $L$ of $K(X)$ such that $M \otimes_R K(X)$ pulled back to $L$ is isomorphic to $L^n$ and such that $\kappa$ corresponds under this isomorphism to the ordinary diagonal Cartier structure on $L^n$. This extension (together with the isomorphism of Cartier modules) lifts to a finite \'etale morphism $g: Z \to U$ on some dense open subset $U$ of $X$. On $Z = \Spec T$ we then have that the pullback of $\kappa$ is given by the diagonal Cartier structure on $g^! M$.
By localizing we may assume that $U = X$. By the same token we may assume that $\omega_X = R$. In particular, $\omega_Z = T$. Consider now the pull back diagram
\[
 \begin{xy}\xymatrix{Z \times_X Y \ar[d]^{g'} \ar[r]^>>>>>{f'}& Z \ar[d]^{g} \\ Y \ar[r]^{f} & X} \end{xy}
\]
and note that both projections are \'etale (cf.\ \cite[Proposition I.3.3]{milne}). Note that $g'$ is surjective since it is finite \'etale and $Y$ is irreducible.
We claim that $(fg')^! M$ is $F$-regular. Then the implication proven above applied to the surjective \'etale morphism $g'$ shows that $f^! M$ is $F$-regular. Since $f'$ is \'etale (hence flat) a generator of $\Hom(F_\ast T, T)$ is again mapped to a generator of $\Hom(F_\ast \mathcal{O}_{Z\times_X Y}, \mathcal{O}_{Z\times_X Y})$ so that the claim follows using Lemmata \ref{FRegularConnComponents} and \ref{RegularimpliesFRegular}.
\end{proof}

\begin{Bem}
The condition that $f$ be surjective is crucial for the implication that $M$ is $F$-regular if $f^! M$ is. This is of course due to the fact that, provided test modules exist, $M$ will always be $F$-regular after a suitable localization. One can also use Lemma \ref{FRegularConnComponents} to show the necessity of surjectivity. 

We suspect that the theorem also holds for more general Cartier algebras. However, we do not know how to prove this except in the special case of Corollary \ref{TestIdealEssEtaleGorenstein} below.
\end{Bem}

\begin{Ko}
\label{TestIdealEtalePullback}
Let $R$ be essentially of finite type over an $F$-finite field, $\mathfrak{a}$ an ideal in $R$ and let $\Spec S \to \Spec R$ be essentially \'etale. Let $M$ be a coherent non-degenerate $F$-regular $\kappa$-module. Then for $t \in \mathbb{Q}_{\geq 0}$ one has $f^! \tau(M, \kappa, \mathfrak{a}^t) = \tau(f^! M, \kappa, (\mathfrak{a}S)^t)$.
\end{Ko}
\begin{proof}
Note that $f^! \mathfrak{a} = \mathfrak{a} \otimes_R S = \mathfrak{a} S$ by flatness. Theorem \ref{FRegularEtalePullback} implies that $f^!(\tau(M, \kappa, \mathfrak{a}^t))$ is $F$-regular with respect to $\kappa^{(\mathfrak{a}S)^t}$ and it clearly generically agrees with $f^!M$. Hence, the claim follows.
\end{proof}

Note that the assumption on $R$ is merely used to ensure that $\tau(M, \kappa, \mathfrak{a}^t)$ exists. If existence is assumed this statement holds in general.

We will call a Cartier algebra $\mathcal{C}$ on $\Spec R$ \emph{locally principal} if there is an open affine covering $(U_i)_i$ of $\Spec R$ such that $\mathcal{C}\vert_{U_i}$ is generated by a single element. A prime example is the Cartier algebra generated by $\Hom_R(F_\ast R, R)$ if $R$ is Gorenstein (cf.\ \cite[Proposition 2.21]{blickleboecklecartierfiniteness} and use duality). 

\begin{Ko}
\label{TestModuleLocallyPrincipal}
Let $R$ be essentially of finite type over an $F$-finite field, $\mathfrak{a}$ an ideal and $t \in \mathbb{Q}$. Furthermore, let $f: \Spec S \to \Spec R$ be essentially \'etale. If $M$ is a non-degenerate $\mathcal{C}$-module where $\mathcal{C}$ is locally principal then $f^! \tau(M, \mathcal{C}^{\mathfrak{a}^t}) = \tau(f^! M, \mathcal{C}'^{(\mathfrak{a}S)^t})$, where $\mathcal{C}'$ is obtained by base change from $\mathcal{C}$.
\end{Ko}
\begin{proof}
Fix an open affine cover $(\Spec R_i)_i$ of $\Spec R$ where $\mathcal{C}$ is principal. Applying Corollary \ref{TestIdealEtalePullback} yields $f^! \tau(M\vert_{R_i}, \mathcal{C}\vert_{R_i}, (\mathfrak{a}R_i)^t) = \tau((f^! M)_{S_i}, \mathcal{C}'\vert_{S_i}, (\mathfrak{a}S_i)^t)$ where $\Spec S_i = f^{-1}(\Spec R_i)$.

Since $f^! \tau(M, \mathcal{C}, \mathfrak{a}^t) \subseteq f^! M$ is a $\mathcal{C}'^{(\mathfrak{a}S)^t}$-module that generically agrees with $f^! \underline{M} = \underline{f^!M}$ (use Lemma  \ref{FPureEtalePullback}) and $\tau(f^!M, \mathcal{C}', (\mathfrak{a}S)^t)$ is minimal with this property we have $\tau(f^!M, \mathcal{C}', (\mathfrak{a}S)^t) \subseteq f^! \tau(M, \mathcal{C}, \mathfrak{a}^t)$. By the previous paragraph they coincide on an open affine cover. So we conclude that equality holds.
\end{proof}

\begin{Ko}
\label{TestIdealEssEtaleGorenstein}
Let $R$ be a Gorenstein domain essentially of finite type over an $F$-finite field, let $\mathcal{C}$ be the Cartier algebra generated by $\Hom(F_\ast R, R)$ and assume that $f: \Spec S \to \Spec R$ is essentially \'etale. Then for $\mathfrak{a}$ an ideal and $t \in \mathbb{Q}$ we have $f^! \tau(R, \mathcal{C}^{\mathfrak{a}^t}) = \tau(S, \mathcal{C}'^{(\mathfrak{a}S)^t})$, where  $\mathcal{C}'$ is the Cartier algebra generated by $\Hom(F_\ast S,S)$.
\end{Ko}
\begin{proof}
As mentioned above $\mathcal{C}$ is locally principal since $R$ is Gorenstein. Now the result folows from Corollary \ref{TestModuleLocallyPrincipal}.
\end{proof}

One can also show that $V$-filtrations in general are well-behaved with respect to $f^!$ for $f$ essentially \'etale:

\begin{Prop}
\label{EssentiallyEtalePullbackVFiltration}
Let $f: \Spec S \to \Spec R$ be an essentially \'etale morphism. Let $\mathfrak{a} \subseteq R$ be an ideal and $M$ a $\kappa$-module on $\Spec R$. Assume that $M$ admits a $V$-filtration along $\mathfrak{a}$. Then $f^! M$ admits a $V$-filtration along $f^! \mathfrak{a}$ and $f^! V^\cdot M = V^\cdot f^! M$.
\end{Prop}
\begin{proof}
By flatness $f^! \mathfrak{a}$ is an ideal in $S$ and the $V^\cdot M \otimes_R S$ form a decreasing sequence of $S$-submodules of $M \otimes_R S$. This filtration is exhaustive, $\mathbb{Q}$-indexed and right-continuous since $V^\cdot M$ is. 
Axioms (i) and (iii) are clear, (ii) follows by flatness and axiom (iv) is clear since $\kappa$ on $f^! M$ is given by $\kappa \otimes \id$.
\end{proof}

\begin{Ko}
Let $f: \Spec S \to \Spec R$ be an essentially \'etale morphism. Let $\mathfrak{a} \subseteq R$ be an ideal and $M$ $\kappa$-module on $\Spec R$. Assume that $M$ admits a $V$-filtration along $\mathfrak{a}$. Then $Gr f^! M = f^! Gr M$.

If in addition $R$ is essentially of finite type over an $F$-finite field, $\mathfrak{a}$ is principal and $M$ is non-degenerate, coherent and $F$-regular so that $Gr M$ and $Gr f^! M$ are endowed with $\kappa$-module structures then $f^! Gr M = Gr f^! M$ also for $\kappa$-modules or crystals. 
\end{Ko}
\begin{proof}
The functor $f^!$ is exact since $f$ is flat. Hence, $f^! Gr M = Gr f^! M$ as $R$-modules. If $R$ is essentially of finite type over an $F$-finite field then Theorems \ref{VfiltrationTestidealfiltration} and \ref{FRegularEtalePullback} imply that $V f^! M$ and $V M$ both coincide with the test modules filtrations along $f^! \mathfrak{a}$ and $\mathfrak{a}$. In particular, one obtains Cartier structures on the associated graded and these are clearly intertwined by the above isomorphism.
\end{proof}

\begin{Theo}
\label{GrEssEtalePullbackCommute}
Let $R$ be essentially of finite type over an $F$-finite field. Let $f: \Spec S \to \Spec R$ be essentially \'etale. Then one has a natural isomorphism of functors $f^! \circ Gr_R = Gr_S \circ f^!$, where $Gr_R$ is considered as a functor from $\mathcal{A}_{\Spec R}$ to $\kappa$-modules on $\Spec R$ and similarly for $Gr_S$. Moreover, we still have an isomorphism of functors if we pass to crystals on the source and target.
\end{Theo}
\begin{proof}
The arguments are similar to those in the proof of Theorem \ref{GrPushforwardCommute}.
\end{proof}

\begin{Bem}
In \cite[Expos\'e XIII, 1.3.7, 1.3.9]{SGA7II} it is shown that for nearby cycles $\Psi$ one has a natural transformation of functors $f^\ast \Psi \to \Psi f^\ast$ and that this is an isomorphism when $f$ is \'etale.
\end{Bem}

Next we prove a partial result for $f^!$, where $f$ is finite flat. Recall that for a finite flat morphism $f: \Spec S \to \Spec R$ and a coherent $R$-module $M$ we have $f^! M = \Hom_R(S, M)$ considered as an $S$-module via premultiplication. If $M$ carries a Cartier structure $\kappa: F_\ast M \to M$ then $f^!M$ is equipped with a Cartier structure by the composition of the natural map $F_\ast \Hom_R(S, M) \to \Hom_R(F_\ast S, F_\ast M)$ with the map $\Hom_R(F_\ast S, F_\ast M) \to \Hom_R(S, M), \varphi \mapsto \kappa \circ \varphi \circ F$ (cf.\ \cite[Remark 3.3.11]{blickleboecklecartiercrystals}).

\begin{Le}
\label{SupportShriekFlat}
Let $f: X \to Y$ be a finite flat morphism and $M$ a coherent $\mathcal{O}_Y$-module. Then $\Supp f^! M = f^{-1} \Supp M$.
\end{Le}
\begin{proof}
Using \cite[Lemma 5.7]{blickleboecklecartierfiniteness} we have $f^! M = f^! \mathcal{O}_Y \otimes_{\mathcal{O}_X} f^\ast M$ and by \cite[Chap.\ 0, Corollaires 5.2.2.2, 5.2.4.1]{EGAISpringer} we obtain $\Supp f^! M = \Supp f^! \mathcal{O}_Y \cap f^{-1} \Supp M$. Hence, we reduced the claim to $\Supp f^! \mathcal{O}_Y = X$.
Let now $\eta$ be a generic point of a component of $X$ and $f(\eta)$ its image which is a generic point of a component of $Y$ due to flatness. Write $i: \Spec \mathcal{O}_{X, \eta} \to X$ and $j: \Spec \mathcal{O}_{Y, f(\eta)} \to Y$. Then we have $i^! f^! = f_{\eta}^! j^!$, where $f_\eta: \Spec \mathcal{O}_{X, \eta} \to \Spec \mathcal{O}_{Y, f(\eta)}$ is induced by $f$. Using the fact that $j^! = j^\ast$ we get that $f_\eta^! j^! (\mathcal{O}_Y) = \Hom_{k(f(\eta))}(k(\eta), k(f(\eta)))$ which is nonzero. Since $i^! f^! \mathcal{O}_Y$ corresponds to $f^! \mathcal{O}_Y$ localized at a generic point of $X$ this shows the claim.
\end{proof}

\begin{Prop}
\label{FlatShriekTestideal}
Let $f: \Spec S \to \Spec R$ be a finite flat morphism, $\mathfrak{a}$ an ideal in $R$ and $M$ a non-degenerate $\kappa$-module and assume that the test module exists for both $M$ and $f^! M$. Then $\tau(f^! M, \kappa^{(\mathfrak{a}S)^t}) \subseteq f^! \tau(M, \kappa^{\mathfrak{a}^t})$.
\end{Prop}
\begin{proof}
We claim that $f^! \tau(M)$ is a $\kappa^{(\mathfrak{a}S)^t}$-module that generically agrees with $f^! M$. Note that $f^! \tau(M)$ carries the structure of a $\kappa^{(\mathfrak{a}S)^t}$-module and recall that $\tau(f^! M)$ is the minimal Cartier submodule that generically agrees with $\underline{f^! M} \subseteq f^!M$. Assuming the claim, we observe that $\mathcal{C}_+^e f^! \tau(M) = \underline{f^! \tau(M)} \subseteq f^! \tau(M)$ generically agrees with $\underline{f^!M}$ since the operation $\underline{\phantom{N}}$ commutes with localization. Hence, it follows by minimality that $\tau(f^!M) \subseteq f^! \tau(M)$.

We now prove the claim. Let $\eta$ be a generic point of $\Supp f^! M$. Then $f(\eta)$ is a generic point of $\Supp M$ by flatness and we have the following commutative diagram:
\[
\begin{xy}
 \xymatrix{\eta \ar[r]^j \ar[d]^{f_\eta} & \Spec R \ar[d]^f \\
f(\eta) \ar[r]^i & \Spec S}
\end{xy}
\]
We have to show that $j^! f^! M = j^! f^! \tau(M)$. Equivalently $ f_\eta^!  i^! M= f_\eta^! i^! \tau(M)$. But by definition of $\tau(M)$ we have $i^! M = i^! \tau(M)$. 
\end{proof}

\begin{Ko}
\label{TraceTestIdeal}
Let $f: \Spec S \to \Spec R$ be a finite flat morphism, $\mathfrak{a}$ an ideal in $R$ and $M$ a $\kappa$-module and assume that the test module exists for both $M$ and $f^! M$. Then the trace map $Tr: f_\ast f^! M \to M$ of Grothendieck-Serre duality restricted to $f_\ast \tau(f^! M)$ factors as $Tr: f_\ast \tau(f^! M) \to \tau(M)$.
\end{Ko}
\begin{proof}
By Proposition \ref{FlatShriekTestideal} above we have an inclusion $f_\ast \tau(f^! M) \subseteq f_\ast f^! \tau(M)$. Since $Tr$ for a finite morphism is given by evaluation at $1$ the claim follows.
\end{proof}

\begin{Bem}
If we consider the Cartier algebra $\kappa^{\mathfrak{a}^t}$ as a set of Cartier linear maps $F^e_\ast M \to M$ for suitable $e$ then these induce via $f^!$ a set of Cartier linear maps $F^e_\ast f^! M \to f^! M$ and the algebra generated by these maps is precisely $\kappa^{(\mathfrak{a}S)^t}$.
\end{Bem}

In the situation of Theorem \ref{VfiltrationTestidealfiltration} it is not clear whether $f^! M$ also admits a $V$-filtration since $f^! M$ will in general not be $F$-regular (cf.\ Example \ref{ShriekNotFpure} below). In order to avoid this problem $Gr$ in the following corollary simply denotes the associated graded of the test module filtration.

\begin{Ko}
Let $R$ be a ring essentially of finite type over an $F$-finite field, $\mathfrak{a}$ a principal ideal and $M$ a $\kappa$-module. If $\Spec S \to \Spec R$ is a finite flat morphism then we obtain a natural $\kappa$-linear morphism $Gr f^! M \to f^! Gr M$.
\end{Ko}
\begin{proof}
From Proposition \ref{FlatShriekTestideal} we obtain a morphism \[Gr f^!M \longrightarrow f^! \tau(M, \kappa^{\mathfrak{a}^{t-\eps}})/f^! \tau(M, \kappa^{\mathfrak{a}^{t}}).\] Since $f^!$ is exact the codomain of this morphism is isomorphic to $f^! Gr M$.
\end{proof}

It is in general not true that the twisted inverse image for a flat morphism preserves $F$-regularity. In particular, the other inclusion of Proposition \ref{FlatShriekTestideal} does not hold in general. This is shown by the following

\begin{Bsp}
\label{ShriekNotFpure}
Consider the finite flat ring extension $R = k[x] \subseteq k[x,y]/(y^2 -x^3) = S$, where $k$ is a field of characteristic $3$ and $R$ is endowed with the ordinary Cartier structure. An $R$-basis of $S$ is given by $\{1, y\}$. Denote the corresponding dual basis of $f^! R = \Hom_R(S, R)$ by $\{\varphi_1, \varphi_y\}$. We claim that $\varphi_y$ does not lie in $\kappa(f^! R)$ (showing that $f^! R$ is not $F$-pure).

Indeed, any element in the image may be written as $\kappa \circ \varphi \circ F$,  where $\varphi = r_1 \varphi_1 + r_y \varphi_y$. Assuming that $\varphi_y$ is of this form we get $\kappa(\varphi(y^{3})) = 1$. Using $y^2 = x^3$ yields $1 = \kappa(r_1 x^3 \varphi_1(y) + r_y x^{3}) = \kappa(r_y x^3) = x \kappa(r_y)$ which is false.

Since $R$ is $F$-regular but $f^! R$ is not even $F$-pure we obtain the desired example.
\end{Bsp}

\section{The connection to unit $F$-modules}
\label{SectionFUnitModules}

In this section we show that for smooth $R$ that Cartier crystals corresponding to locally constant sheaves are contained in the class of modules for which we construct a $V$-filtration. We also discuss possible relations of our notion of $V$-filtration to the one constructed by Stadnik in \cite{stadnikvfiltrationfcrystal}.

We start by briefly reviewing the category of locally finitely generated unit $F$-modules as defined by Emerton and Kisin (cf.\ \cite{emertonkisinintrorhunitfcrys} for an introduction and \cite{emertonkisinrhunitfcrys} for an exhaustive treatment). Let $X$ be a smooth scheme essentially of finite type over an $F$-finite field $k$. A quasi-coherent $\mathcal{O}_X$-module $\mathcal{M}$ that is equipped with an isomorphism $\beta: F^\ast \mathcal{M} \to \mathcal{M}$ is called a \emph{unit $F$-module}. Any unit $F$-module $\mathcal{M}$ can be obtained as the limit of the maps ${F^i}^\ast \beta$ with $\beta: M \to F^\ast M$ and $M$ a quasi-coherent $\mathcal{O}_X$-module. Any such $M$ is called a \emph{root} of $\mathcal{M}$. A unit $F$-module $\mathcal{M}$ is \emph{locally finitely generated (lfgu)} if it admits a coherent root $M$ with $\beta: M \to F^\ast M$ injective. Finally, a unit $F$-module whose underlying $\mathcal{O}_X$-module is coherent is called a \emph{unit $F$-crystal} (in fact, the underlying $\mathcal{O}_X$-module is then locally free -- cf. \cite[Proposition 1.2.3]{emertonkisinintrorhunitfcrys}). Moreover, Emerton and Kisin construct functors $f^!$ and $f_+$ for a morphism $f: X \to Y$ of smooth schemes essentially of finite type over $k$.

The category of Cartier crystals on $X$ is equivalent to the category of lfgu $F$-modules as follows. Given a (minimal) Cartier module $M$ on $X$ tensoring with $\omega_X^{-1}$ yields a structural map $\gamma: M \otimes \omega_X^{-1} \to F^\ast( M \otimes \omega_X^{-1})$. Indeed, tensoring the adjoint structural map $M \to F^!M$ with $\omega^{-1}_X$ and using $F^! \omega_X \cong \omega_X$, $F^! M \otimes (F^! \omega_X)^{-1} \cong F^\ast (M \otimes \omega^{-1}_X)$ (cf.\ \cite[Corollary 5.8]{blickleboecklecartierfiniteness}) induces $\gamma$. Then taking a direct limit of ${F^i}^\ast \gamma$ yields a lfgu $F$-module and this construction induces an equivalence of categories (cf.\ \cite[Theorem 5.15]{blickleboecklecartierfiniteness}). This equivalence of abelian categories induces an equivalence of bounded derived categories. If $M$ is any Cartier module then $M \otimes \omega^{-1}_X$ is a root of the unit $F$-module $\colim {F^e}^\ast M \otimes \omega^{-1}_X$ and if $M$ is a minimal Cartier module then $M \otimes \omega^{-1}_X$ is a minimal root (see \cite{blickleboecklecartiercrystals} and \cite{blicklegammasheaves}). Moreover, this equivalence commutes with $f^!$ and interchanges $(R)f_\ast$ with $(R)f_+$ (see \cite{blickleboecklecartierekequivalence}).

Finally, one obtains an equivalence of the bounded derived category of lfgu $F$-modules with the bounded derived category of constructible $\mathbb{F}_p$-sheaves by first passing to the corresponding \'etale site via the pull back along the natural map $X_{\acute{e}t} \to X_{Zar}$ of sites and then applying the derived $\Hom$ functor in the category of lfgu $F$-modules (see \cite[Sections 9, 11]{emertonkisinrhunitfcrys} for details). Moreover, under this correspondence the abelian category of lfgu $F$-modules is mapped to perverse constructible $\mathbb{F}_p$-sheaves in the sense of Gabber (\cite{gabbertstructures}).

\begin{Le}
If $i: U \to X$ is an open immersion of affine schemes and $M$ is an $F$-regular $\kappa$-module then $i^! M$ is $F$-regular. 
\end{Le}
\begin{proof}
As open immersions are \'etale one has $i^! = i^\ast$. Moreover, $i^\ast M$ is a localization so that this is just Lemma \ref{LeFRegularLocalizes}.
\end{proof}

\begin{Prop}
\label{FunitCrysIsFRegular}
Let $X$ be an affine smooth scheme essentially of finite type over an $F$-finite field. Then any unit $F$-crystal corresponds to an $F$-regular $\kappa$-crystal.
\end{Prop}
\begin{proof}
Let $M$ be a unit $F$-crystal. By definition $M$ is a coherent $\mathcal{O}_X$-module equipped with an isomorphism $M \to F^\ast M$. By \cite[Proposition 1.2.3]{emertonkisinintrorhunitfcrys}, $M$ is locally free. Assume that the corresponding $\kappa$-crystal is not $F$-regular. Hence, there exists a proper subobject $N \subseteq M$ in unit $F$-modules which generically agrees with $M$. Locally on a small open affine (i.\,e.\ after localizing) we have $M \cong R^n$ as well as $N \cong R^n$ and still a strict inclusion $N \subset M$ (just fix a covering of $X$ where both $N$ and $M$ are free then pick an open affine where they do not coincide).

So we have reduced to the situation that $N \subseteq R^n$, where we have a commutative diagram
\[\begin{xy}
   \xymatrix{F^\ast N \ar[r]& F^\ast R^n\\
 N \ar[r] \ar[u]& R^n \ar[u]}
  \end{xy}\]
where the vertical arrows are isomorphism and the horizontal ones are (strict) inclusions and where $N \cong R^n$.

Fixing an isomorphism $\varphi: R^n \to N$ we obtain a commutative diagram

\[\begin{xy}
   \xymatrix{F^\ast N' \ar[r]&F^\ast R^n \ar[r]^{F^\ast\varphi}& F^\ast N\\
 N' \ar[r] \ar[u] & R^n \ar[r]^\varphi \ar[u]& N \ar[u]}
  \end{xy}\] where $N' \to R^n$ is abstractly isomorphic to $N \to R^n$. Note that this yields an inclusion $N' \to N \to R^n$ of unit $F$-crystals. Iterating this construction we obtain an infinite decreasing chain of crystals. This is a contradiction since the category of lfgu $F$-modules has (DCC) (cf. \cite[Proposition 4.2.1]{emertonkisinintrorhunitfcrys}).
\end{proof}

\begin{Bem}
\begin{enumerate}[(a)]
\item{
Let $R$ be as above. What are the $\kappa$-crystals corresponding to locally constant sheaves on $(\Spec R)_{\acute{e}t}$ under this equivalence? A constructible sheaf $\mathcal{F}$ is locally constant if and only if there is a finite surjective \'etale morphism $\varphi: \Spec S \to \Spec R$ such that $\varphi^\ast \mathcal{F}$ is the constant sheaf. Since the constant sheaf corresponds to the unit $F$-module $R^n$ with ordinary Frobenius we find that a $\kappa$-crystal $M$ on $\Spec R$ is locally constant if and only if there is a finite surjective \'etale map $\varphi: \Spec S \to \Spec R$ such that $\varphi^! M = \omega^n$ with $\kappa$-structure given by the natural morphism $F_\ast \omega \to \omega$ on each component.} 
\item{Proposition \ref{FunitCrysIsFRegular} implies that for $X$ a smooth scheme essentially of finite type over an $F$-finite field a locally constant sheaf of finite dimensional $\mathbb{F}_p$-vector spaces on $X_{\acute{e}t}$ corresponds to an $F$-regular Cartier crystal on $X$. Indeed, unit $F$-crystals correspond under the Riemann-Hilbert correspondence of Emerton and Kisin precisely to locally constant sheaves (\cite[Corollary 9.4.2]{emertonkisinrhunitfcrys}, ignoring shifts). In fact, one can use this to give an alternative proof of Proposition \ref{FunitCrysIsFRegular}. The Proposition follows once one shows that a $\kappa$-crystal $M$ such that $\varphi^! M = \omega_S^n$ for $\varphi: \Spec S \to \Spec R$ finite \'etale is $F$-regular. Using Theorem \ref{FRegularEtalePullback} it suffices to show that $\omega_S$ is $F$-regular. Once more applying Theorem \ref{FRegularEtalePullback} and using the fact that \'etale locally $\Spec S$ is covered by $\mathbb{A}^{\dim S}_k$ this boils down to the well-known computation that the Cartier operator on top differential forms of a polynomial ring induces an $F$-regular $\kappa$-structure.}
\item{Since Remark \ref{WrongProp} shows that the pushforward along an open immersion does not preserve $F$-regularity our notion of $V$-filtration does not apply to the objects for which Stadnik constructs a $V$-filtration (\cite{stadnikvfiltrationfcrystal}). Specifically, Stadnik (locally) considers the $V$-filtration along $j_\ast \mathcal{M}$ with respect to $f$, where $\mathcal{M}$ is a unit $F$-crystal on $U = D(f)$ and $j: U \to \Spec R$ denotes the open immersion. Of course, if $j_\ast M$ is the corresponding minimal Cartier module then we may pass to the test module $\tau(j_\ast M)$ with respect to the Cartier-algebra generated by $\kappa$ (which again will be a $\kappa$-module). Then $\tau(j_\ast M)$ corresponds to a unit $F$-module $\mathcal{N} \subseteq j_\ast \mathcal{M}$ and one may now restrict Stadnik's filtration to this submodule and ask how it compares to our $V$-filtration (or rather how the associated graded pieces relate).

Also note that only the zeroth graded piece of Stadnik's filtration carries a unit $F$-structure. So one may optimistically hope that the zeroth graded piece of Stadnik's construction occurs in our $V$-filtration.}
\end{enumerate}
\end{Bem}

\section{$F$-jumping numbers of locally constant $\kappa$-modules}
\label{SectionMainResult}
In this section we will prove that for a $\kappa$-module corresponding to a locally constant sheaf the $F$-jumping numbers of the test module filtration coincide with those of the filtration of the ring (see Theorem \ref{TestIdealLocallyConstant} below). This will enable us to show that for a smooth hypersurface $f$ and $R$ smooth the functor $Gr$ restricted to $F$-regular crystals corresponding to locally constant sheaves is, up to shift, $i^!$, where $i: \Spec R/(f) \to \Spec(R)$ (Corollary \ref{MainResult?}).

Throughout this section rings are $k$-algebras for some fixed $F$-finite field $k$ with some fixed isomorphism $k \to F^! k$. A ring $\varphi:\Spec R \to k$ then comes equipped with a dualizing sheaf $\omega = \varphi^! k$ and an induced isomorphism $\omega \to F^! \omega$. This isomorphism in turn endows $\omega$ with a natural Cartier structure $F_\ast \omega \to \omega$ via its adjoint. If $R$ is smooth then $\omega$ is isomorphic to top dimensional K\"ahler differentials.

\begin{Le}
\label{TraceCartierCommute}
Let $f: X = \Spec S \to Y = \Spec R$ be a finite \'etale morphism of affine integral $F$-finite schemes. Denote by $Tr: S \to R$ the trace map induced by the field extensions of the generic points. Then for any $\kappa$-module $M$ the following diagram commutes
\[
\begin{xy}
 \xymatrix{F_\ast (M \otimes_R S) \ar[d]_{\kappa'} \ar[r]^<<<<<{F_\ast Tr} & F_\ast (M \otimes_R R) \ar[d]^{\kappa}\\
M \otimes_R S \ar[r]^{Tr} & M \otimes_R R}
\end{xy}
\] where $\kappa'$ denotes the unique lift of $\kappa$ constructed in Lemma \ref{CartierStructureEtalePullback} and $Tr = \id \otimes Tr$.
\end{Le}
\begin{proof}
Using $\varphi_S$ from Lemma \ref{CartierStructureEtalePullback} and $\kappa \otimes \id$ we may attach on the left a commutative triangle to our diagram. Similarly, using the isomorphism $\varphi^{-1}_R: F_\ast(M \otimes_R R) \to F_\ast M \otimes_R R$ we may attach a commutative triangle on the right hand side. It is enough to verify that $\kappa F_\ast Tr \varphi_S = Tr (\kappa \otimes \id)$. Indeed, if this is the case then applying $\varphi^{-1}_S$ and using the commutativity of the triangle yields the claim. Also note that we may replace $\kappa$ by $(\kappa_M \otimes \id)\varphi_R^{-1}$.

Now we compute $Tr(\kappa(m \otimes s)) = \kappa(m) \otimes Tr(s)$ and for the left hand side we obtain $\kappa_M \otimes \id(\varphi_R^{-1}(m \otimes F_\ast Tr(s^p)))$. Since, by the proof of \cite[Proposition 4.1]{schwedetuckertestidealsfinitemaps}, we have $Tr(s^p) = Tr(s)^p$ the commutativity follows.
\end{proof}

\begin{Le}
Let $f: \Spec S \to \Spec R$ be a a finite surjective \'etale morphism of affine integral normal schemes. Then the field trace $Tr: Q(S) \to Q(R)$ induces a surjective map $Tr: S \to R$.
\end{Le}
\begin{proof}
Since $f$ is finite surjective \'etale it is faithfully flat. Hence, all the assumptions of \cite[Proposition 7.4]{schwedetuckertestidealsfinitemaps} are satisfied.
\end{proof}

\begin{Le}
\label{TestModuleEtaleTrace}
Let $f: X = \Spec S \to Y = \Spec R$ be a finite surjective \'etale morphism of affine integral normal schemes. Then we have \[f_\ast \tau(f^\ast M, \kappa^{(\mathfrak{a}S)^t}) \cap M \subseteq Tr(f_\ast \tau(f^\ast M, \kappa^{(\mathfrak{a}S)^t)}),\] where $Tr = Tr \otimes \id$. 
\end{Le}
\begin{proof}
Using Corollary \ref{TestIdealEtalePullback} and the remark thereafter we have $(f_\ast \tau(f^! M, \kappa^{(\mathfrak{a}S)^t}) \cap M)S \subseteq f_\ast \tau(f^! M, \kappa^{(\mathfrak{a}S)^t})$. Applying the trace map to the left hand side yields
\begin{align*}Tr((f_\ast \tau(f^! M, \kappa^{(\mathfrak{a}S)^t}) \cap M)S) &= (f_\ast \tau(f^\ast M, \kappa^{(\mathfrak{a}S)^t}) \cap M) \cdot Tr(S)\\ &= (f_\ast \tau(f^! M, \kappa^{(\mathfrak{a}S)^t}) \cap M).\end{align*}
\end{proof}

\begin{Prop}
\label{TestModuleEtaleTransformation}
Let $f: X = \Spec S \to Y = \Spec R$ be a finite \'etale morphism of affine integral normal schemes, $M$ an $F$-pure $\kappa$-module supported on all of $\Spec R$ and $\mathfrak{a} \subseteq R$ an ideal. Suppose furthermore that the test module filtration of $M$ along $\mathfrak{a}$ exists. Then $f_\ast \tau(f^\ast M, \kappa^{(\mathfrak{a}S)^t}) \cap M = \tau(M, \kappa^{\mathfrak{a}^t})$.
\end{Prop}
\begin{proof}
Note that the map $M = M \otimes_R R \to f^\ast M = M \otimes_R S$ is injective (since $R \to S$ is flat). It follows that the above intersection makes sense.

We first prove the inclusion from right to left. Note that $\underline{M}_\eta = M_\eta$, where $\eta$ is the generic point of $\Spec R$. By definition $\tau(M, \kappa^{\mathfrak{a}^t})$ is the smallest $\kappa^{\mathfrak{a}^t}$-submodule of $\underline{M}$ that generically agrees with $M$. Let us denote $f_\ast \tau(f^\ast M, \kappa^{(\mathfrak{a}S)^t}) \cap M$ by $N$. Then $\underline{N}$ is contained in both $\underline{M}$ and $N$. Moreover, since the operation $\underline{\phantom{M}}$ commutes with localization it is thus sufficient to verify that $N_\eta$ coincides with $M_\eta$.

We have $(f_\ast\tau(f^\ast M, \kappa^{{(\mathfrak{a}S)}^t}))_\eta = \tau(f^\ast M, \kappa^{{(\mathfrak{a}S)}^t})_\xi$ considered as $Q(R)$-modules, where $\xi$ denotes the generic point of $\Spec S$. 
By the defining property of test modules the latter equals $M \otimes_R S \otimes_S Q(S) = M \otimes_R Q(S)$ considered as a $Q(R)$-module. The equality $M \otimes_R Q(S) \cap M_\eta = M_\eta$ is clearly satisfied which shows the inclusion.

By Lemma \ref{TestModuleEtaleTrace} it suffices to show that $Tr(f_\ast \tau(f^\ast M, \kappa^{(\mathfrak{a}S)^t})) \subseteq \tau(M, \kappa^{\mathfrak{a}^t})$. By Proposition \ref{FRegularEtalePullback} and \cite[Theorem 3.11]{blicklep-etestideale} we find $c \in R^\circ$ that is a test element for both $M$ and $f^\ast M$. We then have using Lemma \ref{TraceCartierCommute}
\[
 Tr(f_\ast \tau(f^\ast M, \kappa^{(\mathfrak{a}S)^t})) = Tr(\sum_{e \geq 1} \mathcal{C}'_e c \underline{M'}) = \sum_{e \geq 1} \mathcal{C}_e (\id \otimes F_\ast Tr) (c \underline{M'}) = \sum_{e \geq 1} \mathcal{C}_e c \underline{M} \otimes_R R
\]
Under the usual identification of $M \otimes_R R$ with $M$ this yields the claim.
\end{proof}

\begin{Ko}
\label{TestModuleEtaleTransformationNonPrincipal}
Let $R$ be a normal domain essentially of finite type over an $F$-finite field and let $f: X = \Spec S \to Y = \Spec R$ be a finite \'etale morphism with $S$ integral, $M$ an $F$-pure $\mathcal{C}$-module supported on all of $\Spec R$, where $\mathcal{C}$ is a locally principal Cartier algebra, and $\mathfrak{a} \subseteq R$ an ideal. Then $f_\ast \tau(f^! M, \mathcal{C}'^{(\mathfrak{a}S)^t}) \cap M = \tau(M, \mathcal{C}^{\mathfrak{a}^t})$, where $\mathcal{C}'$ is the Cartier algebra obtained by base change.
\end{Ko}
\begin{proof}
The claim may be checked locally on an open affine cover of $\Spec R$, where $\mathcal{C}$ is principal. Now the claim follows from Proposition \ref{TestModuleEtaleTransformation}.
\end{proof}

\begin{Ko}
Let $R$ be a Gorenstein normal domain essentially of finite type over an $F$-finite field and let $f: X = \Spec S \to Y = \Spec R$ be a finite \'etale morphism with $S$ integral and $\mathcal{C}$ the Cartier algebra generated by $\Hom(F_\ast R, R)$. Then $f_\ast \tau(S, \mathcal{C}'^{(\mathfrak{a}S)^t}) \cap R = \tau(R, \mathcal{C}^{\mathfrak{a}^t})$, where $\mathcal{C}' = \Hom(F_\ast S, S)$.
\end{Ko}
\begin{proof}
Follows from Corollary \ref{TestModuleEtaleTransformationNonPrincipal} (cf. also the discussion before Corollary \ref{TestModuleLocallyPrincipal}).
\end{proof}

For a ring $R$ let us denote the Cartier algebra generated by $\Hom(F_\ast R, R)$ by $\mathcal{C}_R$. With this notation we have the following

\begin{Theo}
\label{TestIdealLocallyConstant}
Let $R$ be Gorenstein, $F$-regular and essentially of finite type over an $F$-finite field $k$. Let $M$ be a coherent $\kappa$-module and assume that there is an open affine cover $U_i$ of $\Spec R$ such that there are finite \'etale morphisms $\varphi_i: \Spec S_i \to U_i$ such that $\varphi_i^! M \cong \omega_{S_i}^n$ with the natural Cartier structure on each component. Then $\tau(M, \kappa^{(f)^t}) = \tau(R, \mathcal{C}_R^{(f)^t}) M$ for all $t > 0$.
\end{Theo}
\begin{proof}
It suffices to verify the assertion locally on an open affine cover.
So assume that $(f^! M, \kappa') \cong (\omega_S^n, \kappa)$ for some finite \'etale morphism $f: \Spec S \to \Spec R$. First, we claim that $\tau(f^! M, \kappa'^{{(\mathfrak{a}S)}^t}) = \tau(S, \mathcal{C}_S^{{(\mathfrak{a}S)}^t}) f^!M$. Using the isomorphism above and the fact that $\omega_S$ is locally free we may assume that $\omega_S = S$. By construction $\kappa$ then corresponds to a generator of $\Hom(F_\ast S, S)$ and the claim follows.

Next note that $\tau(S, \mathcal{C}_S^{(\mathfrak{a}S)^t}) f^! M = \tau(R, \mathcal{C}_R^{\mathfrak{a}^t}) f^! M$. Indeed, since $f$ is \'etale Corollary \ref{TestIdealEssEtaleGorenstein} yields that $\tau(R, \mathcal{C}_R^{\mathfrak{a}^t})S = \tau(S, \mathcal{C}_S^{(\mathfrak{a}S)^t})$.

Applying Proposition \ref{TestModuleEtaleTransformation} thus yields \[\tau(M, \kappa^{(f)^t}) = (\tau(R, \mathcal{C}_R^{(f)^t}) (M \otimes_R S)) \cap M = \tau(R, \mathcal{C}_R^{(f)^t}) M.\] The last equality follows from the fact the plus closure of $\tau(R, \mathcal{C}_R^{(f)^t}) M$ is trivial.
Indeed, by \cite[Proposition 8.7]{hochsterhunekebriancon} the tight closure of coherent modules over $R$ is trivial. Since plus closure is contained in tight closure also for modules (\cite[Theorem 5.22]{hochsterhunekesplitting}) the assertion follows.
\end{proof} 

\begin{Bem}
If one is willing to work in $D^b(\Spec R)$ and use dualizing complexes instead of merely sheaves then one should be able to drop the Gorenstein hypothesis in Theorem \ref{TestIdealLocallyConstant} -- it is only used to ensure that $\omega$ is an honest sheaf. 
\end{Bem}

\begin{Def}
Given a Cartier module or crystal satisfying the conditions of Theorem \ref{GrFunctoronCrys} we write $Gr^{[0,1]}M$ for $\bigoplus_{t \in [0,1]} Gr^t M$.
\end{Def}

\begin{Ko}
\label{TestIdealLocallyConstantSmooth}
 Let $R$ be regular and essentially of finite type over an $F$-finite field and let $f \in R$ be a smooth hypersurface.
Let $M$ be a coherent $\kappa$-module and assume that there is an open affine cover $U_i$ of $\Spec R$ such that there are finite \'etale morphisms $\varphi: V_i \to U_i$ such that $\varphi^\ast M$ is trivial (with diagonal Cartier structure). Then $\tau(M, \kappa^{(f)^t}) = f^{\lfloor t \rfloor} M$ for all $t > 0$ and $Gr^{[0,1]} M = M/fM$.
\end{Ko}
\begin{proof}
Since for smooth $f$ one has\footnote{I am not aware of a precise reference. This can be explicitly computed in the case of $f = x_1, R = k[[x_1, \ldots, x_n]]$. The general case reduces to this case by passing to the completion of $R$.} $\tau(R, \kappa^{(f)^{t}}) = (f^{\lfloor t \rfloor})$ this follows from Theorem \ref{TestIdealLocallyConstant}.
\end{proof}

\begin{Ko}
\label{LocallyConstantAssocGraded}
Let $R$ be regular and essentially of finite type over an $F$-finite field and let $f \in R$ be a smooth hypersurface. Let $M$ be an $F$-regular coherent $\kappa$-module whose associated crystal corresponds to a locally constant sheaf on $X_{\acute{e}t}$. Then $Gr^{[0,1]} M = M/fM$ with Cartier structure given by $\kappa f^{p-1}$. If $M = M_{\min}$ then $\tau(M, \kappa^{(f)^t}) = f^{\lfloor t \rfloor} M$ for all $t > 0$.
\end{Ko}
\begin{proof}
Apply Theorem \ref{TestIdealLocallyConstant} and Corollary \ref{TestIdealLocallyConstantSmooth} to $M_{\min}$.
\end{proof}

\begin{Bsp}
\label{ExCrysEssential}
Let $p \geq 3$ be a prime.
Consider $R = \mathbb{F}_p[x] = M$ with Cartier structure $\kappa x$, where $\kappa$ is a generator of $\Hom(F_\ast R, R)$ and $f = x$.

Then we claim that $\tau(R, (x)^t) = (x)$ for $\frac{p-2}{p-1} \leq t \leq 1$. Note that $(\kappa x)^e = \kappa^e x^{\frac{p^e-1}{p-1}}$ and that $M$ is $F$-regular (use \cite[Proposition 3.7]{blicklep-etestideale} -- here we need $p \geq 3$, otherwise $M$ is $F$-pure but not $F$-regular). In order to compute the test ideal we may localize at $x$ so that $R_x$ is $F$-regular with respect to $\kappa x$. Then using \cite[Theorem 3.11]{blicklep-etestideale} we obtain that \[\tau(R, (x)^{t}) = \sum_{e \geq 1} (\kappa x)^e x^{\lceil t(p^e -1)\rceil} x R = \sum_{e \geq 1} \kappa^e x^{\frac{p^e-1}{p-1}} x^{\lceil t(p^e-1) \rceil}x R = \tau(R, x^{t + \frac{1}{p-1}}),\] if in the last step we endow $R$ again with its ordinary Cartier structure. Thus the claim follows.

Also note that $M = M_{\min}$. It is $F$-pure as we have seen above. Moreover, $M$ has no nontrivial (nilpotent) $\kappa$-submodules. Indeed, since $M = R$ any nonzero $\kappa$-submodule has to contain $\tau(M, \kappa x) = R$ and $R$ is not nilpotent. However, the Cartier structure on $Gr^{[0,1]} M$ is nilpotent, so that the associated crystal is zero.

Finally, we are interested in the constructible sheaf $\mathcal{F}$ to which $M$ corresponds on $X_{\acute{e}t}$, where $X = \Spec R$. Since $R = k[x]$ the dualizing sheaf is just $\omega_R$ (which we may identify with $R$). A computation shows that the corresponding lfgu $F$-module is isomorphic to $k[x,x^{-1}]$. Then one computes (using \cite[1.3.3 and 3.2]{emertonkisinintrorhunitfcrys}) that the corresponding constructible sheaf $\mathcal{F}$ is given by $\mathcal{F}\vert_{V(x)} = 0$ and $\mathcal{F}\vert_{D(x)} = \mathbb{F}_p$.
\end{Bsp}

In order to properly state our next result we need to work with derived categories. Recall that for a finite morphism $f: Y \to X$ the twisted inverse image functor $f^!(-)$ in $D(X)$ is naturally isomorphic to $f^{-1} R\mathcal{H}om(f_\ast \mathcal{O}_Y, -)$. In particular, if $i: \Spec R /(f) \to \Spec R$ is a closed immersion then $i^!(-) = RHom(R/(f), -)$ and if $f$ is a regular element of $R$ then $i^! M = M/fM[-1]$.

\begin{Le}
\label{CartierStructureShriek}
Let $R$ be an $F$-finite regular noetherian ring and $f \in R$ a regular element. Let $M$ be a coherent $\kappa$-module and denote $i: \Spec R/(f) \to \Spec R$. Then the Cartier module structure on $(i^! M)[1]$ as defined in \cite[Remark 3.3.11]{blickleboecklecartiercrystals} is given by \begin{align*}F_\ast (M/fM) &\longrightarrow M/fM,\\m + fM &\longmapsto \kappa(f^{p-1} m) + fM.\end{align*}
\end{Le}
\begin{proof}
See \cite[Example 3.3.12]{blickleboecklecartiercrystals}.
\end{proof}

\begin{Ko}
\label{MainResult?}
Assume that $R$ is regular, essentially of finite type over an $F$-finite field and that $f \in R$ defines a smooth hypersurface.
Consider the functor $Gr^{[0,1]}$ from the category $\mathcal{A}$ of $F$-regular $\kappa$-crystals for which $f$ is a regular element to the category of $\kappa$-crystals on $\Spec R/(f)$. Then $Gr^{[0,1]}$ restricted to the full subcategory $\mathcal{A}_{lc}$ of $\mathcal{A}$ whose objects correspond to crystals that are trivialized by a finite \'etale morphism is naturally isomorphic to the restriction of $i^![1]$ to $\mathcal{A}_{lc}$, where $i: \Spec R/(f) \to \Spec(R)$ is the natural closed immersion.
\end{Ko}
\begin{proof}
This follows from Lemma \ref{CartierStructureShriek} and Corollary \ref{TestIdealLocallyConstantSmooth}.
\end{proof}

We can rephrase the above corollary using the equivalence with constructible sheaves on $X_{\acute{e}t}$ if we assume that $R$ is smooth. Namely, the crystals in $\mathcal{A}_{lc}$ are precisely those crystals in $\mathcal{A}$ which correspond to a locally constant sheaf on $X_{\acute{e}t}$. As such Corollary \ref{MainResult?} may be seen as an analog of \cite[2.1.5]{SGA7II} -- see however Remark \ref{GrNotNearbyCycles} below.

\begin{Que}
Let $R$ be a regular ring essentially of finite type over an $F$-finite field and let $\mathfrak{a} = (f)$ be a \emph{smooth} hypersurface. Let $M$ be an $F$-regular non-degenerate $\kappa$-module. Is it then true that the crystals $Gr^{[0,1]} M$ and $i^! M$ are isomorphic?
\end{Que}

\begin{Bem}
\begin{enumerate}[(a)]
 \item{
Since the modules for which we constructed a $V$-filtration always satisfy $Gr^0(M) = 0$ this question may be seen as an analog of \cite[III, Proposition 4.5.3]{mebkhoutdxmodules}. Also note that Example \ref{ExCrysEssential} shows the passage to crystals is essential.}
\item{If $f$ is not smooth the functor $Gr$ is much more complicated, cf.\ \cite[Proposition 4.2]{mustatawatanabetestvsmultiplier} where the test ideal filtration is worked out for certain singular hypersurfaces. Despite the result above, we cannot expect to have an exact triangle $i^! M[1] \to Gr^{[0,1]}M \to Gr^{[0,1)}M$ as is the case in characteristic zero when constructing $V$-filtrations.}
\end{enumerate}

\end{Bem}

\begin{Bem}
\label{GrNotNearbyCycles}
In closing we would like to point out that the functor $Gr$ constructed in Theorem \ref{GrFunctoronCrys} does \emph{not} correspond to nearby cycles on $X_{\acute{e}t}$. This is due to the fact that the (derived) nearby cycles functor is not a functor on $D_c(X_{\acute{e}t})$ in the case where one considers a scheme $X$ over $\mathbb{F}_p$ and $\mathbb{F}_p$ coefficients. Also note, that the functor $Gr$ behaves for locally constant sheaves and $f$ a smooth hypersurface much like the case $\ell \neq p$ -- cf.\ \cite[2.1.5]{SGA7II}

One may therefore still ask whether $Gr$ corresponds maybe to the unipotent part of the (tame) nearby cycles functor (cf.\ \cite[Expos\'e I, 2.7]{SGA7I} for a definition of tame nearby cycles) or whether the functor $Gr$ preserves the desirable properties of nearby cycles in the $\ell \neq p$ case while avoiding the pathologies of the $\ell = p$ case. We plan to pursue these and related questions in future work.
\end{Bem}

\bibliography{bibliothek.bib}
\bibliographystyle{amsplain}
\end{document}